\documentclass[11pt]{article}       
\usepackage{geometry}               
\usepackage{graphicx}
\usepackage{geometry}
\usepackage{caption}
\usepackage{subcaption}
\usepackage{mathrsfs}
\usepackage{comment}
\usepackage{amsmath,esint} 
\usepackage{amssymb}  
\usepackage{amsthm}  
\usepackage{bm}
\usepackage{color}
\usepackage{array}
\usepackage{cite}
\usepackage{hyperref}
\usepackage{comment}

\newtheorem{proposition}{Proposition}
\newtheorem{defn}{Definition}
\newtheorem{corollary}{Corollary}
\newtheorem{lemma}{Lemma}
\newtheorem{theorem}{Theorem}
\newtheorem{remark}{Remark}

\numberwithin{equation}{section}

\newcommand{\sleq}{\leqslant}
\newcommand{\sgeq}{\geqslant}
\newcommand{\ep}{\varepsilon}

\newcommand{\E}{\mathbb{E}}

\newcommand{\N}{\mathbb{N}}

\newcommand{\R}{\mathbb{R}}
\newcommand{\bS}{\mathbb{S}}

\newcommand{\Q}{\mathbb{Q}}

\newcommand{\cB}{\mathcal{B}}

\newcommand{\cF}{\mathcal{F}}
\newcommand{\cH}{\mathcal{H}}

\DeclareMathOperator*{\argmin}{arg\,min}

\title{Existence of flat flows for volume-preserving mean curvature flow with contact angle} 

\newcommand{\footremember}[2]{%
    \footnote{#2}
    \newcounter{#1}
    \setcounter{#1}{\value{footnote}}%
}

\author{
  Jiwoong Jang \footremember{trailer}{Department of Mathematics, University of Maryland-College Park (jjang124@umd.edu)},
  }

\date{}
\providecommand{\keywords}[1]{\textbf{Key words.} #1}
\providecommand{\MSC}[1]{\textbf{MSC codes.} #1}

\begin{document}
\maketitle

\pagestyle{myheadings}
\thispagestyle{plain}

\begin{abstract}

We study the motion of a droplet evolving by mean curvature with volume constraint and contact angle condition on a half space. We prove the existence of a global-in-time weak solution, called the flat flow. A difficulty arises when we establish the local-in-time equi-boundedness of approximate solutions and a uniform $L^2$-estimate of multipliers. The difficulty is handled by conducting blowup analysis at a point in contact to a spherical cap with sharp angle. 
\end{abstract}


\vspace{0.2cm}

\MSC{35D30, 49J53, 53E10, 58E12.}

\section{Introduction}\label{sec:introduction}
We let and fix throughout this paper an integer $d\sgeq1$, $\Omega:=\R^{d+1}$, and a Lipschitz function $\beta\,:\,\partial\Omega\to[-1+2\kappa,1-2\kappa]$ for a fixed $\kappa\in\left(0,\frac12\right]$. We consider a family $\{E_t\}_{t\sgeq0}$ of subsets of $\Omega$. The family $\{E_t\}_{t\sgeq0}$ is said to evolve by \emph{the volume-preserving mean curvature flow with contact angle condition $\beta$} if $E_t$ is open and smooth for each $t\sgeq0$, and
\begin{equation}
\begin{cases}
v=-H_{_{E_t}}+\lambda\qquad\quad\,\,\text{on }\partial E_t\cap\Omega\qquad\qquad\text{for }t>0,\\
\nu_{E_t}\cdot(-e_{d+1})=\beta \qquad\text{on }\partial_{\partial\Omega}(\partial E_t\cap\partial\Omega)\quad\,\text{for }t\sgeq0.
\end{cases}\label{eq:VPMCF-contact-angle}
\end{equation}
Here, $v$ is the normal speed on $\partial E_t\cap\Omega$ (taken negative for shrinking sets), $H_{E_t}$ is the mean curvature vector of $\partial E_t\cap\Omega$ (taken positive for convex sets), $\nu_{E_t}$ is the outer unit normal vector of $\partial E_t$, $e_{d+1}:=(0,\cdots,0,1)\in\R^{d+1}$, $\partial_{\partial\Omega}(\partial E_t\cap\partial\Omega)$ is the topological boundary of $\partial E_t\cap\partial\Omega$ in $\partial\Omega$. The number $\lambda$ is the multiplier responsible for the volume constraint given by
\begin{align*}
\lambda=\lambda(t)=\frac{1}{\mathcal{H}^{d}(\partial E_t\cap\Omega)}\int_{\partial E_t\cap\Omega}H_{E_t}\, d\cH^{d}.
\end{align*}
By the volume constraint we mean that with the $\lambda$ above, the total volume is preserved over time so that $|E_t|=|E_0|$ for $t\sgeq0$.

\vspace{\medskipamount}
The model \eqref{eq:VPMCF-contact-angle} is primarily motivated from the motion of a liquid drop on a flat surface whose first study goes back to Young \cite{Y1805} and Laplace \cite{L1805}. The surface (the closure of $\partial E_t\cap\Omega$ in the above notations) between the vapor and the liquid is called \emph{a capillary surface}, and the triple junction ($\partial_{\partial\Omega}(\partial E_t\cap\partial\Omega)$) where the vapor, liquid, and solid meet is called \emph{a contact line}. The coefficient function $\beta\,:\,\partial\Omega\to[-1+2\kappa,1-2\kappa]$ models the relative adhesion coefficient of the three media.

In \cite{Y1805}, the effect of mean curvature to the surface is first introduced (and followed by \cite{L1805}), and the contact angle condition for a stationary capillary surface on the contact line is formulated, now known as \emph{Young's law}. Later, Gauss introduced in \cite{G1830} a free energy of a (smooth) region $E$ occupied by the liquid, with volume constraint $|E|=m_0$ for a number $m_0>0$ that is fixed throughout the paper to denote the initial volume, given by
\[
\mathcal{F}(E) = P(\partial E \cap \Omega) + \int_{\partial E \cap \partial \Omega} \beta(x') \, d\mathcal{H}^{d}(x') + \int_E g(x) \, dx,
\]
where $P(\cdot)$ is the perimeter of the argument set, $g\,:\,\Omega\to\R$ is a potential energy from the gravity per unit mass.

The motion of the liquid will be seen as the gradient flow of the total energy $\cF$, that is, the movement which reduces the energy in the fastest manner. Due to possible singularities in a finite time notably from the mean curvature effect, we consider in this paper the movement in a generalized sense that will be explained in a moment, whose existence is the main subject of this paper. In this paper, we take $g=0$ in order to focus effects from the other components that are the mean curvature, the angle condition, and the volume constraint. It is expected, in view of the existence of a generalized motion, that the study in the gravity-free case can easily be adapted to the case with gravity, which we do not fully include here. Finally, we leave \cite{M08,F86} for historical accounts and modern developments of capillary surfaces, and we also refer to \cite{BK18} and the references therein for further and various interdisciplinary motivations.



\vspace{\medskipamount}
Even if we start with an initial datum regular enough, mean curvature flows, with or without volume constraint and contact angle conditions, are known to develop singularities such as topological changes in a finite time. Several weak notion of solutions to mean curvature flows that are defined beyond the first singularity time have been suggested, for which we give an outline (by no means complete) in the following.

Weak notions based on the viscosity solution theory to the pure mean curvature flow, where moving surfaces are regarded as a level-set of a value function that serves as an unknown, are given independently in \cite{ES91,CGG91} with basic properties. With contact angle condition in this framework, and with a forcing term given a priori, the weak notion through viscosity solutions is studied in \cite{B99,IS04,JKMT22,J23}.

Approximations to the mean curvature motion through the called Allen-Cahn equation, both for volume constraint and for contact angle condition, have a rich body of literature. Focusing on volume-preservation condition, \cite{T23} verifies in all dimensions the convergence to a weak solution to the motion, the so-called $L^2$-flows developed in \cite{MR08}, of solutions to Allen-Cahn-type equations to Allen-Cahn equations with a soft penalization term for volume constraint \cite{MSS16}. The Allen-Cahn equation considered in \cite{T23} subsequently yields applications recently; one \cite{J25} which takes hard obstacles into account while keeping volume constraint and another \cite{P25} which provides a varifold solution that is proved to both exist and agree with the strong solution (if the latter exist). The notion in \cite{P25} is based on the one in \cite{HL24B} for the mean curvature motion. See \cite{HL24B} and the references therein for recent works on the weak-strong uniqueness that is studied extensively.


The mean curvature flow with contact angle condition also has been studied in the framework of Allen-Cahn equations. When the prescribed angle is identically the right-angle along the boundary, the convergence to a weak notion upto the boundary is proved on a convex domain \cite{MT15} and later generalized on a possibly nonconvex domain \cite{K19}. Generalizing the results in this regard, especially the motion law upto the boundary, for general angles has been left open to the best of the author's knowledge. Under a natural energy assumption (see \eqref{eq:energy-assumption}) which is first formulated in \cite{LS95}, \cite{HL24} establishes the existence of BV-distributional solutions and the weak-strong uniqueness.

\vspace{\medskipamount}

We take the approach based on minimizing movements, independently proposed by \cite{ATW93} and \cite{LS95}. The works \cite{ATW93,LS95} consider an alternative proxy (see \eqref{eq:energies}) instead of the $L^2$-distance from the metric tensor due to the complete degeneracy \cite{MM06}. Any limit flow obtained as time step goes to zero is called the flat flow. The volume-preservation condition is handled in \cite{MSS16} with a penalization term, while the contact angle condition for the flat flow is studied in \cite{BK18}. In particular for the latter, the advantage of considering flat flows in \cite{BK18} compared to the Allen-Cahn formulation is that the motion law is concluded upto the boundary for general angles. We combine the approaches from \cite{MSS16,BK18} in this paper.

A technical difficulty arises when we prove boundedness of sets (denoted by $E^h_{t}$ with time step size $h\in(0,1)$, see Section \ref{sec:prelim-main-results}) from minimizing movements uniformly in time step size which is needed to ensure the convergence to flat flows. The work \cite{BK18} for the mean curvature flow with contact angle is based on the comparison principle, which we cannot expect to hold for the volume-preserving mean curvature flow due to the non-locality.

By minimizing movements, we obtain a sequence of sets $\{E_k\}_{k\in\N_{\sgeq0}}$, where $E_k$ is a minimizer of a suitable functional with a reference set $E_{k-1}$ for $k\sgeq1$. To prove the equi-boundedness (that is, uniform in time step size) of the sets, the work \cite{MSS16} considers a ball that shrinks starting with a large radius until the ball meets the set $E_k$ at a contact point $x_0$. Thanks to the minimality of $E^k$ and the classical study on (interior) regularity of minimal surfaces, the point $x_0$ is proved to a regular point in the sense that the surface $\partial E^k$ is a $C^{2,\alpha}$ manifold near $x_0$. This observation allows to connect the distance, say $r_k$, of the set $E^k$ from the origin and the distance $r_{k-1}$ of its reference set $E_{k-1}$, which yields a useful recursive relation of the distances $\{r_k\}_{k\in\N_{\sgeq0}}$. As this argument relies on the interior regularity of minimal surfaces, it does not directly work for our problem in case the point $x_0$ is on the boundary $\partial\Omega$.

In this regard, we consider the fundamental work \cite{DePM15} that establishes the boundary regularity of minimizers for a capillary functional and proves (anisotropic) Young's law for the minimizers. Now thanks to the validity of Young's law, we then naturally consider spherical caps that meet $\partial\Omega$ with ``sharp angles," in the expectation that the spherical cap while shrinking does not meet the set $E_k$ on the boundary $\partial\Omega$. This expectation rigorously works indeed when the contact point, say $x_0$, now on $\partial\Omega$ is a regular point so that Young's law can be applied. However, as we do not know a priori if $x_0$ is regular or not unlike in the interior case as the above, we need to perform a careful blowup analysis to derive a contradiction at the contact point $x_0$ that can possibly be singular. After achieving to derive a contradiction, we finish proving the equi-boundedness as the interior regularity theory applies as the above.

\vspace{\medskipamount}





Once flat flows for \eqref{eq:VPMCF-contact-angle} exist as in Theorem \ref{thm:flat-flows}, we can ask next questions that are important from physical motivations. A recent work \cite{JN24} proves the consistency based on the uniform ball condition for the problem under volume constraint but without contact angle, while \cite{K25} proves the consistency based on the comparison principle for the contact angle problem but without volume constraint. As the two approaches seem not applicable to the settings of each other, establishing the consistency for our problem \eqref{eq:VPMCF-contact-angle} with the both conditions is an interesting question.

The large-time behavior for flat flows and quantitative versions of the Alexandrov theorem have recently been studied actively. See \cite{JMPS23} and the references therein. Recently, \cite{JZ25} establishes a quantitative Alexandrov theorem for capillary surfaces, and as a further application of the result, \cite{JZ25} remarks the large-time behavior of flat flows for our model \eqref{eq:VPMCF-contact-angle} as a future research problem.

\vspace{\medskipamount}
The paper is organized as follows: Section \ref{sec:prelim-main-results} arranges notations, definitions, and settings, and state the main result (Theorem \ref{thm:flat-flows}). Section \ref{sec:approximate-solutions} states and proves properties of approximate flows such as the existence of minimizers at each step, density estimates, and the Euler-Lagrange equation. Assuming Proposition \ref{prop:barrier-functions}, we prove Theorem \ref{thm:flat-flows}  at the end of Section \ref{sec:approximate-solutions}. Proposition \ref{prop:barrier-functions} is proved in Section \ref{sec:barrier-functions}. We then prove Theorem \ref{thm:BV-distributional-solutions} in Section \ref{sec:BV-solutions} under the assumption \ref{eq:energy-assumption}.

Throughout the rest of the paper, we will denote positive constants that may vary line by line by several characters such as $C,c,R,\theta>0$ and possibly more. Dependency of these numbers on quantities such as $d,\kappa,T>0$ will be specified as arguments.

\section{Preliminaries and main results}\label{sec:prelim-main-results}
For a Lebesgue measurable set $E\subset\mathbb{R}^{d+1}$, $\chi_{E}$ denotes the characteristic function of $E$, and $|E|$ denotes its Lebesgue measure. For an open set $U\subset\mathbb{R}^{d+1}$, the set of $L^1(U)$-functions of bounded variation is denoted by $BV(U)$, and we let $BV(U,\{0,1\})$ be the set of all sets of finite perimeter in $U$, that is,
\[
BV(U,\{0,1\}):=\{E\subset U\,:\,\chi_E\in BV(U)\}.
\]
For $E\in BV(U,\{0,1\})$, $P(E,U)$ denotes the relative perimeter of $E$ in $U$, that is,
\[
P(E, U) := \sup \left\{\left.\int_{E} \operatorname{div} \varphi \,d x\,: \,\varphi \in C_{c}^{1}\left(U ; \mathbb{R}^{d+1}\right) \text { with } \sup _{U}|\varphi| \leq 1\right\},\right.
\]
and we let $P(E):=P(E,\mathbb{R}^{d+1})$. We denote by $\partial^{\ast}E$ the essential boundary of $E$ and by $\nu_{E}(x)$ the measure-theoretic outward unit normal vector of $E$ at $x\in\partial^{\ast}E$. Among Lebesgue equivalent sets in $U$ for a given Lebesgue measurable set $E\subset U$, we consider the set of points of density one as the representative of $E$. Similar definitions can be made for open subsets of $\R^{d}$ which is identified with $\partial\Omega$. Here (and henceforth), we fix $\Omega:=\mathbb{R}^{d+1}\times(0,\infty)$. We write a vector $x\in\mathbb{R}^{d+1}$ in a way that $x=(x',x_{d+1})\in\mathbb{R}^d\times\mathbb{R}$. Similarly, $\nabla'$ ($\mathrm{div}'$, resp.) denotes the gradient (the divergence, resp.) in the variable $x'\in\R^d$. The notation $e_{d+1}$ refers to the vector $(0,\cdots,0,1)\in\R^{d+1}$.

\vspace{\medskipamount}

We say in this paper that a vector field $\Psi=(\Psi',\Psi_{d+1})\in C^1_c(\overline{\Omega},\R^{d+1})$ is admissible if $\Psi_{d+1}=0$ on $\partial\Omega$.

\vspace{\medskipamount}
We denote by $\mathrm{Tr}(E)$ the trace set of $E\in BV(\Omega,\{0,1\})$ on $\partial\Omega$. We also write $\chi_{\mathrm{Tr}(E)}=\chi_E$ on $\partial\Omega$ by abuse of notations, and we set
\[
P(E, \overline{\Omega}) := P(E, \Omega)+\int_{\partial \Omega} \chi_{E} \,d \mathcal{H}^{d},
\]
which agrees with $P(E)$. We say that a set $E\in BV(\Omega,\{0,1\})$ has a generalized mean curvature $H_E\in L^1(\partial^{\ast}E\cap\Omega,d\mathcal{H}^{d})$ in $\Omega$ if it holds
\[
\int_{\partial^{*} E} \operatorname{div}_{\partial E} \Psi \,d \mathcal{H}^{d}=\int_{\partial^{*} E} \Psi \cdot \nu_{E} H_{E} \,d \mathcal{H}^{d} \quad \text { for all } \Psi \in C_{c}^{1}\left(\Omega , \mathbb{R}^{d+1}\right).
\]
Here, $\operatorname{div}_{\partial E} \Psi$ denotes the tangential divergence of $\Psi$ defined on $\partial^{\ast}E$ by $\operatorname{div}_{\partial E} \Psi :=\operatorname{div} \Psi-\nu_{E} \cdot \nabla \Psi \nu_{E}$. The sign convention we take here is that $H_E\geq0$ for a convex set $E$. Regarding more definitions and properties of $BV$-functions, we refer to \cite{G84}.

\vspace{\medskipamount}

For a given measurable set $F\subset\R^{d+1}$, we denote by $\operatorname{sd}_F$ the signed distance from $F$ being negative inside $F$, that is,
\[
\operatorname{sd}_{F}(x):=\left\{\begin{array}{ll}
\operatorname{dist}(x, F) & \text { for } x \in F^{c}, \\
-\operatorname{dist}(x, F^{c}) & \text { for } x \in F.
\end{array}\right.
\]
Here, $\operatorname{dist}(x, G)$ is defined by $\inf_{y\in G}|x-y|$ for a set $G\subset\R^{d+1}$. We set $\operatorname{d}_F:=|\operatorname{sd}_F|$.

\vspace{\medskipamount}

Throughout the paper, we fix $m_0>0$, $\kappa\in\left(0,\frac{1}{2}\right]$, and $\beta\in \mathrm{Lip}(\partial\Omega)$ with $\|\beta\|_{L^{\infty}(\partial\Omega)}\sleq1-2\kappa$. For $h\in(0,1)$ and $E,F\in BV(\Omega,\{0,1\})$, we define the capillary energy $C_{\beta}(E)$ and the capillary version of Almgren--Taylor--Wang-version functional $\mathcal{F}^h(E,F)$ by
\begin{align}
\begin{cases}
\,\, \quad C_{\beta}(E):=P(E, \Omega)+\int_{\partial \Omega} \beta  \chi_{E} \,d \mathcal{H}^{d},\\
\mathcal{F}^{h}(E, F):=C_{\beta}(E)+\frac{1}{h} \int_{E \Delta F} \operatorname{d}_F\,d x+\frac{1}{\sqrt{h}}\left||E|-m_{0}\right|.
\end{cases}\label{eq:energies}
\end{align}

\vspace{\medskipamount}

We define approximate flat flows and flat flows for \eqref{eq:VPMCF-contact-angle} as follows. The integer which is the greatest among all integers smaller or equal to $s\in\R$ is denoted by $\lfloor s\rfloor$. 

\begin{defn}
Let $E_0\in BV(\Omega,\{0,1\})$ be a bounded set of volume $m_0$, and let $h\in(0,1)$.
\begin{itemize}
    \item[(i)] (Approximate flat flows) Let $\{E^h_{kh}\}_{k\in\N_{\geq0}}$ be a sequence of sets in $BV(\Omega,\{0,1\})$ defined recursively by $E^h_0=E_0$ and
    \[
    E^h_{kh}\in\argmin \left\{\mathcal{F}^h(E,E^h_{(k-1)h})\,:\,E\in BV(\Omega,\{0,1\})\right\}\qquad\text{for }k\sgeq1.
    \]
    We call $\{E^h_t\}_{t\sgeq0}$ with $E^h_t:=E^h_{\lfloor t/h\rfloor h}$ for $t\sgeq0$ an approximate flat flow for \eqref{eq:VPMCF-contact-angle} with initial datum $E_0$.
    \item[(ii)] (Flat flows) Any $L^1$-limit of approximate flat flows is called a flat flow.
\end{itemize}
\end{defn}
\vspace{\medskipamount}

We now present our main results.

\begin{theorem}[Existence of flat flows]\label{thm:flat-flows}
Let $E_0\in BV(\Omega,\{0,1\})$ be a bounded set of volume  $m_0$, and let $\{E^h_{t}\}_{t\sgeq0}$ be an approximate flat flow for \eqref{eq:VPMCF-contact-angle} with initial datum $E_0$ for each $h\in(0,1)$. Then, there exist a family of sets $\{E_t\}_{t\sgeq0}$ in $BV(\Omega,\{0,1\})$ and a sequence  $h\to0$ such that
\[
|E_t^{h} \Delta E_t| \to 0\quad \text{as }h\to0 \quad \text{for } t \sgeq0.
\]
Moreover, there exists $\theta(d,\kappa)>0$ such that for all $t,s\sgeq0$, we have $|E_t|=m_0$ and $$|E_t\Delta E_s|\sleq \theta(d,\kappa)P(E_0)|t-s|^{1/2}.$$
\end{theorem}

We now make the assumption \eqref{eq:energy-assumption} as proposed in \cite{LS95} and obtain the compatibility of flat flows with BV-distributional solutions. We recall that a vector field $\Psi=(\Psi',\Psi_{d+1})\in C^1_c(\overline{\Omega},\R^{d+1})$ is called admissible if $\Psi_{d+1}=0$ on $\partial\Omega$.

\begin{theorem}[Compatibility with BV-distributional solutions]\label{thm:BV-distributional-solutions}
Let the assumptions and the notations in Theorem \ref{thm:flat-flows} be in force. Suppose that
\begin{align}\label{eq:energy-assumption}
\mathcal{H}^d\lfloor_{\partial^{\ast}E^h_t\cap\Omega}\quad\longrightarrow{}\quad\mathcal{H}^d\lfloor_{\partial^{\ast}E_t\cap\Omega}\qquad\text{weakly$^{\ast}$ as $h\to0$ for a.e. $t\sgeq0$}.
\end{align}
Then, the following hold:
\begin{itemize}
    \item[(i)] There exist functions $v:\Omega\times(0,\infty)\to\R$, $\lambda:(0,\infty)\to\R$, $H_E:\Omega\times(0,\infty)\to\R$ and constants $C(d,\kappa),\ R_T=R_T(d,\kappa,T)>0$ (for a given $T>0$) such that
\begin{equation}\label{eq:L2-target-functions}
\left\{
\begin{aligned}
&\qquad\qquad\int_0^\infty \int_{\partial^{\ast}E_t\cap\Omega} v^2 \, d\mathcal{H}^d dt \leq C(d,k)P(E_0), \\
&\int_0^T \int_{\partial^{\ast}E_t\cap\Omega} H_E^2 \, d\mathcal{H}^d dt + \int_0^T \lambda^2 dt \leq R_T \quad \text{for each } T > 0,
\end{aligned}
\right.
\end{equation}
and the set $\partial E_t\cap\Omega$ has a generalized mean curvature $H_{E}(\cdot,t)$ for a.e. $t\sgeq0$. Moreover,
    

    \item[(ii)] For admissible $\Psi\in C^1_c\left(\overline{\Omega},\R^{d+1}\right)$ and a.e. $t\sgeq0$, it holds that 
    \begin{align}
    \int_{\partial^{\ast}E_t\cap\Omega} (\mathrm{id}-\nu_{E_{t}} \otimes \nu_{E_{t}}) :& \nabla \Psi\,d\mathcal{H}^d + \int_{\partial^{\ast}E_t\cap\Omega}v \Psi \cdot v_{E_{t}} \,d\mathcal{H}^{d}\notag\\
    &+ \int_{\mathrm{Tr}(E_{t})} \mathrm{div}'(\beta \Psi') \,d\mathcal{H}^{d-1} = \lambda(t) \int_{\partial^{\ast}E_t\cap\Omega} \Psi\cdot\nu_{E_{t}} \,d\mathcal{H}^{d}.\label{eq:without-perimeter-assumption}    
    \end{align}
    
    \item[(iii)] If $1\sleq d\sleq 6$, a flat flow $\{E_t\}_{t\sgeq0}$ is a BV-distributional solution to the motion law $v=-H_E+\lambda$ with initial datum $E_0$, that is,
    \begin{align}
    \int_0^\infty \int_{E_t} \phi_t \, dx \, dt + \int_{E_0} \phi(x,0) \, dx &= - \int_0^\infty \int_{\partial^{\ast}E_t\cap\Omega} \phi v \, d\mathcal{H}^ddt\label{eq:motion-law-v}\\
    &=\int_0^\infty \int_{\partial^{\ast}E_t\cap\Omega} \phi (H_{E}-\lambda) \, d\mathcal{H}^ddt\label{eq:motion-law-lambda}
    \end{align}
    for $\phi\in C^1_c(\Omega\times[0,\infty))$. Also, it holds that for a.e. $t\sgeq0$,
    \begin{align}\label{eq:lambda}
        \lambda(t)=\frac{1}{\mathcal{H}^d(\partial^{\ast}E_t\cap\Omega)}\int_{\partial^{\ast}E_t\cap\Omega}H_E(\cdot,t)d\mathcal{H}^d.
    \end{align}

\end{itemize}
\end{theorem}

\section{Properties of approximate flat flows}\label{sec:approximate-solutions}
\subsection{Minimizers at discrete time steps}\label{subsec:minimizers}
First of all, we state that there is a solution to the minimization problem for the functional $\cF^{h}$, which in particular implies that of approximate flat flows with bounded initial datum of finite perimeter. For $r>0$, let $B_r$ ($B'_r$, resp.) denote the ball in $\R^{d+1}$ (in $\R^d$, resp.) with center the origin and radius $r$.

\begin{proposition}[\!\!\protect{\cite[Theorem 4.1, Appendix A]{BK18}}, \cite{CM07}]\label{prop:existence-minimizers}
Let $F\in BV(\Omega,\{0,1\})$ be a bounded set with $D,H>0$ taken so that the cylinder $B'_D\times(0,H)$ contains $F$. Then, the minimization problem
\[
\min\left\{\cF^h(E,F)\,:\,E\in BV(\Omega,\{0,1\})\right\}
\]
has a solution. Moreover, any minimizer of the problem is contained $B'_{R_0}\times(0,H+1)$, where
\[
R_0:=D+1+\max\left\{8^{d^2+d+1}\left(\kappa^{-1}\left(C_{\beta}(F)+\frac{1}{\sqrt{h}}||F|-m_0|\right)\right)^{\frac{d+1}{d}},4\left(\kappa^{-1}+2\right)^{\frac{d+1}{d}}\right\}.
\]
\end{proposition}


We state the normalization and the boundary regularity results for minimizers from \cite{DePM15}.

\begin{theorem}[\!\!\protect{\cite[Theorem 1.10, Lemma 2.16]{DePM15}}]\label{thm:DeMP}
Under the settings and the notations of Proposition \ref{prop:existence-minimizers}, let $E\in BV(\Omega,\{0,1\})$ be a nonempty minimizer. Then, the following hold:
\begin{itemize}
\item[(i)] Upto a Lebesgue null set, $E$ is open in $\R^{d+1}$, and $\partial^{\ast}E\cap\Omega$ is a $d$-dimensional $C^{2,\alpha}$-manifold for some $\alpha\in(0,1)$. Moreover, it holds that
\[
\cH^{s}((\partial E\setminus\partial^{\ast}E)\cap\Omega)=0\qquad\text{for }s>d-7.
\]
\item[(ii)] The set $\partial E \cap \partial \Omega$ satisfies $\mathcal{H}^{d}((\partial E \cap \partial \Omega)\Delta(\mathrm{Tr}(E)))=0.$
\item[(iii)] Let $M=\overline{\partial E\cap\Omega}$. Then, $\partial_{\partial\Omega}(\partial E \cap \partial \Omega) = M \cap \partial \Omega$, where $\partial_{\partial\Omega}(\partial E \cap \partial \Omega)$ denotes the topological boundary of $\partial E \cap \partial \Omega$ in $\partial\Omega$.
\item[(iv)] There exists a closed subset $\Sigma$ of $M$ with $\cH^{d-1}(\Sigma\cap\partial\Omega)=0$ such that the set $M$ is a $C^{1,1/2}$-manifold with boundary around each point in $(M\cap\partial\Omega)\setminus\Sigma$. Moreover, Young's law holds true on $(M \cap \partial \Omega) \setminus \Sigma$, that is,
\[
\nu_{E} \cdot (-e_{d+1}) = \beta \quad \text{on} \quad (M \cap \partial \Omega) \setminus \Sigma.
\]
\end{itemize}
\end{theorem}

Approximate flat flows enjoy dissipation inequalities as follows.

\begin{proposition}\label{prop:dissipations}
Let $E_0\in BV(\Omega,\{0,1\})$ be a bounded set of volume  $m_0$, and let $\{E^h_{t}\}_{t\sgeq0}$ be an approximate flat flow for \eqref{eq:VPMCF-contact-angle} with initial datum $E_0$. Then, for all $k\in\N$,
\begin{itemize}
\item[(i)] it holds that
\begin{align*}
C_{\beta}(E^h_{kh})+\frac{1}{h}\int_{E^h_{kh}\Delta E^h_{(k-1)h}}\mathrm{d}_{E^h_{(k-1)h}}\,dx&+\frac{1}{\sqrt{h}}\left|\left|E^h_{kh}\right|-m_0\right|\\
&\sleq C_{\beta}(E^h_{(k-1)h})+\frac{1}{\sqrt{h}}\left|\left|E^h_{(k-1k)h}\right|-m_0\right|.
\end{align*}
\item[(ii)] Also, we have, for all $k\in\N$,
\begin{align*}
\begin{cases}
C_{\beta}(E^h_{kh})+\frac{1}{\sqrt{h}}\left|\left|E^h_{kh}\right|-m_0\right|\sleq P(E_0),\\
\sum_{s=1}^k\frac{1}{h}\int_{E^h_{sh}\Delta E^h_{(s-1)h}}\mathrm{d}_{E^h_{(s-1)h}}\,dx\sleq P(E_0).
\end{cases}
\end{align*}
\end{itemize}
\end{proposition}
\begin{proof}
The statement (i) is obtained from the relation $\cF^h(E^h_{kh},E^h_{(k-1)h})\sleq\cF^h(E^h_{(k-1)h},E^h_{(k-1)h})$, while (ii) follows from iterations of (i), Lemma \ref{lem:coercivity}, and $|E_0|=m_0$.
\end{proof}

\subsection{Density estimates and their consequences}
Before we begin density estimates, we first state the coercivity of perimeters in the capillary functional in the following lemma.

\begin{lemma}[\!\!\protect{\cite[Proposition 3.2]{BK18}}]\label{lem:coercivity}
For any $E\in BV(\Omega,\{0,1\})$, it holds that
\begin{align*}
\kappa P(E)\sleq C_{\beta}(E)\sleq P(E).
\end{align*}
\end{lemma}

In the rest of this subsection, we fix $k\in\N$ and put $E=E^h_{kh},\, F=E^h_{(k-1)h}$ for simplicity.

\begin{proposition}[\!\!\protect{\cite[Proposition 5.5]{BK18}}]\label{prop:L-infty-density}
With $R=R(d,\kappa):=4(d+1)\left(\frac{2}{\kappa}\right)^{\frac{d+1}{2}}$, it holds that
\[
\|\mathrm{d}_F\|_{L^{\infty}(E\Delta F)}\sleq R\sqrt{h}.
\]
\end{proposition}
\begin{proof}
Let $x\in E\setminus F$, and let $B_s$ denote $B_s(x)$ in this proof. It suffices to show
\[
\mathrm{sd}_F(x)\sleq R\sqrt{h}
\]
as the other case can be proven similarly. Assume that
\begin{align}\label{eq:too-far}
\mathrm{sd}_F(x)>(R+\ep)\sqrt{h}\qquad\text{for some }\ep>0.
\end{align}
Put $\rho:=\frac{1}{2}(R+\varepsilon)\sqrt{h}$, and let $r\in(0,\rho)$. The comparison $\cF^h(E,F)\sleq\cF^h(E\setminus B_r,F)$ gives
\begin{align}\label{eq:comparison}
P(E,B_r\cap\Omega)+\int_{B_r\cap\partial\Omega}\beta\chi_{E}\,d\cH^d+\frac{1}{h}\int_{E\cap  B_r}\mathrm{sd}_{F}\,dx\sleq\cH^d(E\cap \partial B_r)+\frac{1}{\sqrt{h}}|E\cap B_r|.    
\end{align}
Moreover, we note, from \eqref{eq:too-far}, that
\[
\frac{1}{\sqrt{h}}|E\cap B_r|-\frac{1}{h}\int_{E\cap  B_r}\mathrm{sd}_{F}\,dx\sleq\frac{1}{\sqrt{h}}\left(1-\frac{R+\varepsilon}{2}\right)|E\cap B_r|.
\]
Let $\Lambda':=\frac{1}{\sqrt{h}}\left(1-\frac{R+\varepsilon}{2}\right)<0$. Therefore, by Lemma \ref{lem:coercivity} and \eqref{eq:comparison}, we get
\begin{align}
\kappa P(E \cap B_r) &\sleq P(E, B_r \cap \Omega) + \cH^d(E \cap \partial B_r) + \int_{B_r\cap\partial\Omega} \beta\chi_E d\cH^d \notag\\
&\sleq 2\cH^d(E \cap \partial B_r)+\Lambda'|E\cap B_r|.\label{eq:too-far-result}
\end{align}
By the isoperimetric inequality, we have
\[
(d+1)\omega_{d+1}^{\frac{1}{d+1}}|E\cap B_r|^{\frac{d}{d+1}}\sleq P(E\cap B_r).
\]
Setting $m(r):=|E\cap B_r|$ for $r\in(0,\rho)$, we obtain, together with \eqref{eq:too-far-result}, that
\[
\frac{1}{2}\kappa\omega_{d+1}^{\frac{d}{d+1}}m(r)^{\frac{d}{d+1}}\sleq m'(r)\qquad\text{for a.e. }r\in(0,\rho).
\]
Consequently, $|E\cap B_r|=m(r)\sgeq\left(\frac{\kappa}{2}\right)^{d+1}\omega_{d+1}r^{d+1}$ for $r\in[0,\rho]$. Using \eqref{eq:too-far-result} and $\Lambda'<0$, we get
\begin{align*}
(-\Lambda')\left(\frac{\kappa}{2}\right)^{d+1}\omega_{d+1}r^{d+1}\sleq(-\Lambda')|E\cap B_r|&\sleq 2\cH^d(E\cap \partial B_r)\\
&\sleq 2\cH^d(\partial B_r)= 2(d+1)\omega_{d+1}r^d\quad\text{for }r\in[0,\rho],
\end{align*}
and therefore,
\[
(-\Lambda')\left(\frac{\kappa}{2}\right)^{d+1}r\leq2(d+1)\qquad\text{for }r\in[0,\rho].
\]
Substituting $r=\left(\frac{R+\varepsilon}{2}-1\right)\sqrt{h}\in[0,\rho]$, we get
\begin{align*}
\left(\frac{R+\varepsilon}{2}-1\right)\sqrt{h}\sleq(-\Lambda')^{-1}\left(\frac{2}{\kappa}\right)^{d+1}2(d+1)\sleq\frac{\sqrt{h}}{\frac{R+\varepsilon}{2}-1}\left(\frac{2}{\kappa}\right)^{d+1}2(d+1),
\end{align*}
and in turn,
\[
R+\varepsilon\sleq2\left(\frac{2}{\kappa}\right)^{\frac{d+1}{2}}\sqrt{2(d+1)}+2.
\]
A contradiction occurs from the choice $R=4(d+1)\left(\frac{2}{\kappa}\right)^{\frac{d+1}{2}}$.
\end{proof}

\begin{corollary}[\!\!\!\protect{\cite[Corollary 3.3]{MSS16}}]\label{cor:density}
There exists $c=c(d,\kappa)\in(0,1)$ such that for $x\in\partial E,\, r\in\left(0,c\sqrt{h}\right)$, it holds that
\begin{align*}
\begin{cases}
\min\left\{|B_r(x)\setminus E|,\,|B_r(x)\cap E|\right\}\sgeq cr^{d+1}, \\
\qquad \qquad  c \sleq \frac{P(E, B_r(x))}{r^d} \sleq c^{-1}.
\end{cases}
\end{align*}
\end{corollary}
\begin{proof}
The relation $\cF^h(E,F)\sleq\cF^h(E\setminus B_r,F)$ for a fixed $x\in \partial E,\,B_r=B_r(x)$ yields, by Proposition \ref{prop:L-infty-density} and a similar argument in its proof, that
\[
\kappa P(E\cap B_r)\sleq 2\cH^d(E\cap \partial B_r)+\frac{R+1}{\sqrt{h}}|E\cap B_r|
\]
where $R:=4(d+1)\left(\frac{2}{\kappa}\right)^{\frac{d+1}{2}}$. Running a similar argument with $m(r):=|E\cap B_r|$, we obtain that there exists a constant $c=c(d,\kappa)\in(0,1)$ such that
\[
|E\cap B_r|\sgeq cr^{d+1}\qquad\text{for }r\in\left(0,c\sqrt{h}\right).
\]
The statement that $|B_r\setminus E|\sgeq cr^{d+1}$ is obtained from the relation $\cF^h(E,F)\sleq\cF^h(E\cup B_r,F)$.

For perimeters, the statement $P(E,B_r)\sgeq cr^{d}$ follows from the isoperimetric inequality. With a sufficiently small $c=c(d,\kappa)\in(0,1)$, the inequality $P(E,B_r)\sleq c^{-1}r^d$ follows from
\begin{align*}
P(E,B_r)&=P(E,\Omega\cap B_r)+\int_{\partial\Omega\cap B_r}\chi_E\,d\cH^d\\
&\sleq \cH^d (E \cap \partial B_r) + \int_{\partial\Omega\cap B_r} (1-\beta) \chi_E\,d\cH^d + \frac{R+1}{\sqrt{h}} |E \cap B_r|\\
&\sleq (d+1) \omega_{d+1} r^d + 2 \omega_d r^d + \frac{R+1}{\sqrt{h}} \omega_{d+1} r^{d+1}\\
&\sleq c^{-1}r^d\qquad\qquad\text{for }r\in\left(0,c\sqrt{h}\right).
\end{align*}
\end{proof}

\begin{proposition}[\!\!\protect{\cite[Proposition 5.7]{BK18}, \cite[Proposition 3.4]{MSS16}}]\label{prop:L1}
Let $c=c(d,\kappa)\in(0,1)$ be the constant as in Corollary \ref{cor:density}. Then, there exists a constant $C(d,\kappa)>1$ such that
\[
|E\Delta F|\sleq C(d,\kappa)P(F)l+\frac{1}{l}\int_{E\Delta F}\mathrm{d}_F\,dx\qquad\text{for }l\in\left(0,c\sqrt{h}\right).
\]
\end{proposition}
\begin{proof}
We omit the proof as it is obtained similarly to that of \cite[Proposition 5.7]{BK18}.
\end{proof}

Let $E_0\in BV(\Omega,\{0,1\})$ be a bounded set of volume  $m_0$, and let $\{E^h_t\}_{t\sgeq0}$ be an approximate flat flow for \eqref{eq:VPMCF-contact-angle} with initial datum $E_0$.

\begin{proposition}[\!\!\protect{\cite[(7.2)]{BK18}, \cite[Proposition 3.5]{MSS16}}]\label{prop:C-half-approximate-flow}
There exists a constant $\theta(d,\kappa)>1$ such that
\[
\left|E^h_t\Delta E^h_s\right|\sleq\theta(d,\kappa)P(E_0)|t-s|^{1/2}\quad\text{whenever }t>s\sgeq0\text{ and }h\in\left(0,\frac{1}{2}\min\{1,{4c^2(t-s)}\}\right).
\]
Here, $c=c(d,\kappa)\in(0,1)$ is the constant as in Corollary \ref{cor:density}.
\end{proposition}
\begin{proof}
Fix $0\sleq s<t$, and let $j(s):=\lfloor s/h\rfloor,\, j(t):=\lfloor t/h\rfloor$, and $N:=j(t)-j(s)\in\N_{\sgeq0}$. 
We write
\begin{align}
\left|E^h_t\Delta E^h_s\right|=\left|E^h_{j(t)h}\Delta E^h_{j(s)h}\right|&\sleq\sum_{j=j(s)}^{j(t)-1}\left|E^h_{(j+1)h}\Delta E^h_{jh}\right|\notag\\
&\sleq\sum_{j=j(s)}^{j(t)-1}\left(C(d,\kappa)P(E^h_{jh})l+\frac{1}{l}\int_{E^h_{(j+1)h}\Delta E^h_{jh}}\mathrm{d}_{E^h_{jh}}\,dx\right)\label{eq:L1-applied}
\end{align}
for any $l\in\left(0,c\sqrt{h}\right)$. By Lemma \ref{lem:coercivity} and Proposition \ref{prop:dissipations}, we have
\begin{align}\label{eq:dissipation-applied}
\sum_{j=j(s)}^{j(t)-1}P(E^h_{jh})\sleq\kappa^{-1}P(E_0)N\,\quad\text{and}\quad\sum_{j=j(s)}^{j(t)-1}\int_{E^h_{(j+1)h}\Delta E^h_{jh}}\mathrm{d}_{E^h_{jh}}\,dx\sleq hP(E_0).
\end{align}
Combining \eqref{eq:L1-applied} and \eqref{eq:dissipation-applied}, we obtain
\begin{align}\label{eq:L1-applied-2}
\left|E^h_t\Delta E^h_s\right|\sleq \frac{C(d,\kappa)P(E_0)}{\kappa}Nl+\frac{hP(E_0)}{l}.
\end{align}
Take $l:=\frac{1}{2}|t-s|^{-\frac12}h$, which then belongs in the open interval $\left(0,c\sqrt{h}\right)$ due to the condition $h\in\left(0,\frac{1}{2}\min\{1,{4c^2(t-s)}\}\right)$. It is easy to see that
\begin{align*}
Nh|t-s|^{-1/2}\sleq(t-s+h)|t-s|^{-1/2}\sleq(1+2c^2)|t-s|^{1/2}.
\end{align*}
Substituting the above in \eqref{eq:L1-applied-2}, we obtain
\[
\left|E^h_t\Delta E^h_s\right|\sleq\left(\kappa^{-1}C(d,\kappa)(1+2c^2)+2\right)P(E_0)|t-s|^{1/2},
\]
which completes the proof with $\theta(d,\kappa):=\kappa^{-1}C(d,\kappa)(1+2c^2)+2$.
\end{proof}

\subsection{$L^2$-estimates}\label{subsec:L2-estimates}

Set
\[
v^{h}(x,t) := \begin{cases} \frac{1}{h} \mathrm{sd}_{E_{t-h}^{h}}(x) & \text{for } t \in [h, \infty), \\ 0 & \text{for } t \in [0, h). \end{cases}
\]

\begin{proposition}[\!\!\protect{\cite[Lemma 8.9]{BK18}, \cite[Lemma 3.6]{MSS16}}]\label{prop:L2-discrete-velocities}
There exists a constant $A(d,\kappa)>1$ such that
\[
\int_0^{\infty}\int_{\partial E^h_t\cap\Omega}|v^h|^2\,d\cH^ddt\sleq A(d,\kappa)P(E_0).
\]
\end{proposition}
\begin{proof}
We omit the proof as it is obtained similarly to that of \cite[Lemma 8.9]{BK18}.
\end{proof}

\begin{proposition}[\!\!\protect{\cite[Proposition 8.2]{BK18}, \cite[Lemma 3.7]{MSS16}}]\label{prop:EL-equation}
Let $E_0\in BV(\Omega,\{0,1\})$ be a bounded set of volume  $m_0$, and let $\{E^h_t\}_{t\sgeq0}$ be an approximate flat flow for \eqref{eq:VPMCF-contact-angle} with initial datum $E_0$. Let $\Psi=(\Psi',\Psi_{d+1})\in C^1_c(\overline{\Omega},\R^{d+1})$ be an admissible vector field. Then, it holds, for $t\sgeq h$, that
\begin{align}
\int_{\partial^{\ast}E^h_t\cap\Omega} \mathrm{div}_{\partial E^h_t}&\Psi\,d\cH^d + \frac{1}{h} \int_{\partial^{\ast}E^h_t\cap\Omega} \mathrm{sd}_{E^h_{t-h}}\Psi\cdot\nu_{E^h_t}\,d\cH^d\notag\\
&+\int_{\partial^{\ast}\mathrm{Tr}(E^h_t)}\beta\Psi'\cdot\nu'_{\mathrm{Tr}(E^h_t)}\,d\cH^{d-1} = \lambda^h(t) \int_{\partial^{\ast}E^h_t\cap\Omega}  \Psi\cdot\nu_{E^h_t}\,d\cH^d,\label{eq:discrete-EL}
\end{align}
where $\nu'_{\mathrm{Tr}(E^h_t)}$ is the measure theoretic outer unit normal vector of $\partial^{\ast}\mathrm{Tr}(E^h_t)$ in $\partial\Omega$. In case $|E^h_t|=m_0$, the number $\lambda^h\in\R$ is a Lagrange multiplier under the volume constraint, and in the other case $|E^h_t|\neq m_0$, we have $\lambda^{h}(t)=\frac{1}{\sqrt{h}}\mathrm{sgn}(m_0-|E^h_t|)$. 
\end{proposition}

\begin{remark}\label{rmk:regular-points}
The equation
\begin{align}\label{eq:pointwise}
H_{E^h_t}+v^h(\cdot,t)=\lambda^h(t)
\end{align}
holds in a distributional sense on $\partial E^h_t\cap\Omega$ and in a point-wise sense on $\partial^{\ast}E^h_t\cap\Omega$. In particular, Proposition \ref{prop:EL-equation}, the Lipschitz regularity of $v^h$, and standard Schauder estimates yield the $C^{2,\alpha}$-regularity of $\partial^{\ast}E^h_t\cap\Omega$ (which already is stated in Theorem \ref{thm:DeMP}) and in turn having the equation hold true point-wisely on $\partial^{\ast}E^h_t\cap\Omega$.
\end{remark}

\begin{proposition}\label{prop:barrier-functions}
Let $E_0\in BV(\Omega,\{0,1\})$ be a bounded set of volume  $m_0$ contained in $B'_{R_0}\times(0,H_0)$ for $R_0,H_0>0$, and let $\{E^h_t\}_{t\sgeq0}$ be an approximate flat flow for \eqref{eq:VPMCF-contact-angle} with initial datum $E_0$. Then, for a given $T>0$, there exist $h_0=h_0(m_0,P(E_0))\in(0,1)$, $R_T=R_T(d,\kappa,m_0,P(E_0),R_0,H_0,\|\nabla'\beta\|_{L^{\infty}(\partial\Omega)},T)>0$ such that
\begin{itemize}
    \item[(i)] it holds $E^h_t\subset B_{R_T}$ for $t\in[0,T],\, h\in(0,h_0)$, and that
    \item[(ii)] it holds
    \[
    \int_0^T|\lambda^h(t)|^2\,dt\sleq R_T\qquad\text{for }h\in(0,h_0).
    \]
    Here, $\lambda^h(t)$ is as in Proposition \ref{prop:EL-equation} for $t\sgeq h$, and $\lambda^h(t)$ is set to be $0$ for $t\in[0,h)$.
\end{itemize}
\end{proposition}

\begin{corollary}\label{cor:L2-mean-curvature}
For a given $T>0$, there exist constants $h_0=h_0(m_0,P(E_0))\in(0,1)$, $R_T=R_T(d,\kappa,m_0,P(E_0),R_0,H_0,\|\nabla'\beta\|_{L^{\infty}(\partial\Omega)},T)>0$ such that for $h\in(0,h_0)$,
\begin{itemize}
    \item[(i)] it holds
    \[
    \int_0^T\int_{\partial E^h_t\cap\Omega}H_{E^h_t}^2\,d\cH^ddt\sleq R_T,\qquad\text{and}
    \]
    \item[(ii)] the set $\Sigma:=\left\{t\in[0,T]\,:\,|E^h_t|\neq m_0\right\}$ satisfies $|\Sigma|\sleq R_Th$. In particular,
    \[
    \left|\left\{n\in\N_{\sgeq0}\,:\,|E^h_{nh}|\neq m_0,\, nh\sleq T\right\}\right|\sleq R_T.
    \]
\end{itemize}
\end{corollary}
\begin{proof}
Claim (i) follows from Proposition \ref{prop:L2-discrete-velocities}, Remark \ref{rmk:regular-points}, and Proposition \ref{prop:barrier-functions} (with a possibly larger $R_T$ than the one in Proposition \ref{prop:barrier-functions}). Claim (ii) follows from Proposition \ref{prop:barrier-functions} and the fact that $\lambda^{h}(t)=\frac{1}{\sqrt{h}}\mathrm{sgn}(m_0-|E^h_t|)$ when $|E^h_t|\neq m_0$ with the inequality
\[
|\Sigma| \leq h \int_{0}^{T} |\lambda^{h}(t)|^{2} dt \leq R_Th.
\]
\end{proof}

\subsection{Proof of Theorem \ref{thm:flat-flows}}\label{subsec:existence-flat-flows}
We present the proof based on the propositions developed in Section \ref{sec:approximate-solutions}. Here, we assume Proposition \ref{prop:barrier-functions} and postpone the proof until Section \ref{sec:barrier-functions}. A subsequence of $h\to0$ will be denoted by $h\to0$ by abuse of notations.

\begin{proof}[Proof of Theorem \ref{thm:flat-flows}]
Let $E_0\in BV(\Omega,\{0,1\})$ be a bounded set of volume  $m_0$, and let $\{E^h_t\}_{t\sgeq0}$ be an approximate flat flow for \eqref{eq:VPMCF-contact-angle} with initial datum $E_0$. 
From Lemma \ref{lem:coercivity}, Proposition \ref{prop:dissipations}(ii), and Proposition \ref{prop:barrier-functions}(i), we can find a subsequence $h\to0$ and a $BV(\Omega,\{0,1\})$-family $\{E_t\}_{t\in\Q_{\sgeq0}}$ such that $E_t\bigr\rvert_{t=0}=E_0$, $E_t\subset B_{R_T}$ for $t\in[0,T]\cap\Q$, $|E_t|=m_0$ for $t\in\Q_{\sgeq0}$, and
\[
\left|E^h_t\Delta E_t\right|\to0\qquad\text{as }h\to0\text{ for }t\in\Q_{\sgeq0}.
\]
Moreover, Proposition \ref{prop:C-half-approximate-flow} yields
\[
|E_t\Delta E_s|\sleq\theta(d,\kappa)P(E_0)|t-s|^{1/2}\qquad\text{for }t,s\in\Q_{\sgeq0}.
\]
By the completeness of $L^1(\Omega)$, $\{E_t\}_{t\in\Q_{\sgeq0}}$ uniquely extends to a $BV(\Omega,\{0,1\})$-family $\{\E_t\}_{t\sgeq0}$ such that
\[
|E_t\Delta E_s|\sleq\theta(d,\kappa)P(E_0)|t-s|^{1/2}\qquad\text{for }t,s\sgeq0.
\]
It is clear that $|E_t|=m_0$ for all $t\sgeq0$. The $L^1(\Omega)$-convergence of $E^h_t$ to $E_t$ as $h\to0$ for general $t\sgeq0$ can be seen by taking $t_{\varepsilon}\in (t-\varepsilon,t+\varepsilon)\cap\Q_{\sgeq0}$ for a given $\varepsilon\in(0,1)$, and by the observation that
\[
\limsup_{h\to0}\left|E^h_t\Delta E_t\right|\sleq\limsup_{h\to0}\left|E^h_t\Delta E^h_{t_{\varepsilon}}\right|+\limsup_{h\to0}\left|E^h_{t_{\varepsilon}}\Delta E_{t_{\varepsilon}}\right|+\left|E_{t_{\varepsilon}}\Delta E_{t}\right|\sleq2\theta(d,\kappa)P(E_0)\varepsilon^{1/2}.
\]
As $\varepsilon\in(0,1)$ was arbitrary, we obtain
\[
\left|E^h_t\Delta E_t\right|\to0\qquad\text{as }h\to0\text{ for }t\sgeq0.
\]
We complete the proof.
\end{proof}

\section{Proof of Proposition \ref{prop:barrier-functions}}\label{sec:barrier-functions}
Let $C_r=B\cap\Omega$ be the part of the ball $B\subset\R^{d+1}$ in $\Omega$ whose base $C_r\cap\partial\Omega$ on $\partial\Omega$ is the ball centered at the origin in $\partial\Omega$ of radius $r>0$ and whose outer unit normal vector $\nu_{C_r}$ of $\partial C_r$ at $\partial C_r\cap\partial\Omega$ satisfies
\[
\nu_{C_r}\cdot(-e_{d+1})=-(1-\kappa)\left(<-\|\beta\|_{L^{\infty}(\partial\Omega)}\right),\quad\text{or},\quad\nu_{C_r}\cdot e_{d+1}=1-\kappa\left(>\|\beta\|_{L^{\infty}(\partial\Omega)}\right).
\]
The following will be crucial in proving Proposition \ref{prop:barrier-functions}.

\begin{proposition}\label{prop:spherical-caps}
For $t\sgeq 0$, set $r_t:=\inf\{r>0\,:\,E^h_t\subset C_r\}$. If $t\sgeq h$, $\partial C_r$ does not meet $\partial E^h_t$ at $\partial E^h_t\cap\partial\Omega$, that is,
\[
\partial C_{r_t}\cap\partial E\cap \partial\Omega=\emptyset.
\]
Therefore, the spherical cap $\partial C_{r_t}$ meets $\partial E^h_t$ in $\Omega$.
\end{proposition}

\begin{remark}\label{rmk:regular-points-with-spherical-caps}
As a consequence of Proposition \ref{prop:spherical-caps}, the set $\partial C_{r_t}\cap\partial E^h_t\cap\Omega$, now a nonempty set, is contained in $\partial^{\ast}E^h_t\cap\Omega$. Otherwise, by looking at blowup at a contact point, we would have a minimizing cone contained in a half space which is absurd. We indeed will conduct a similar blowup analysis now with contact angle when proving Proposition \ref{prop:spherical-caps}.
\end{remark}

\vspace{\medskipamount}

\subsection{Proposition \ref{prop:spherical-caps} implies Proposition \ref{prop:barrier-functions}}\label{subsec:prop9-implies-prop10}
We first prove Proposition \ref{prop:barrier-functions} assuming Proposition \ref{prop:spherical-caps}. We show Proposition \ref{prop:spherical-caps} in the next subsection.

\begin{proof}[Proof of Proposition \ref{prop:barrier-functions}]
For $t\sgeq h$, by Proposition \ref{prop:spherical-caps} and Remark \ref{rmk:regular-points-with-spherical-caps}, the equation $H_{E^{h}_t}+v^{h}(\cdot,t)=\lambda^{h}(t)$ classically holds on the set $\partial C_{r_t}\cap\partial E^h_t\cap\Omega$ of contact points, and at the same time, $H_{E^h_t}\sgeq0$ on the set, which yields
\[
r_t\sleq r_{t-h}+h|\lambda^h|.
\]
Iterating the same process gives
\begin{align}\label{eq:lambda-bounds-r}
r_t\sleq r_0+\int_0^t|\lambda^h(\tau)|\,d\tau.
\end{align}

We can check that by the construction of $C_{r_t}$ with contact angle condition (the inner product being $\pm(1-\kappa)$) and by the condition $\kappa\in\left(0,\frac12\right]$, we have $C_{r_t}\subset B'_{r}\times(-2r,2r)\subset B_{5\sqrt{3}r}=B_{10r_t}$ with $r=\frac{2}{\sqrt{3}}r_t$. Choose an admissible vector field $\Psi\in C^1_c(\overline{\Omega},\R^{d+1})$ that is $\Psi=\mathrm{id}$ in $B_{10r_t}$. Then, $|\Psi|\sleq 10r_t$ in $C_{r_t}$. Meanwhile, for small $h$ depending on $m_0,\,P(E_0)$, we have $|E^h_t|\sgeq\frac12m_0$ from Proposition \ref{prop:dissipations}. We estimate each term in the Euler-Lagrange equation in Proposition \ref{prop:EL-equation}.

Applying Lemma \ref{lem:coercivity}, we obtain
\begin{align}\label{eq:term1-in-EL}
\int_{\partial^{\ast}E^h_t\cap\Omega} \mathrm{div}_{\partial E^h_t}\Psi\,d\cH^d&=dP(E^h_t,\Omega)\sleq d\kappa^{-1}P(E_0),
\end{align}
and by the Cauchy-Schwarz inequality,
\begin{align}\label{eq:term2-in-EL}
\int_{\partial^{\ast}E^h_t\cap\Omega} v^h\Psi\cdot\nu_{E^h_t}\,d\cH^d&\sleq 10r_tP(E^h_t,\Omega)^{1/2}\|v^h\|_{L^2(\partial^{\ast}E^h_t\cap\Omega,d\cH^d)}.
\end{align}
Thanks to Lemma \ref{lem:coercivity}, the boundary term satisfies, for smooth functions $\beta$, and thus for Lipschitz functions $\beta$ by approximation, that
\begin{align}
\left|\int_{\partial^{\ast}\mathrm{Tr}(E^h_t)}\beta\Psi'\cdot\nu'_{\mathrm{Tr}(E^h_t)}\,d\cH^{d-1}\right|&=\left|\int_{\mathrm{Tr}(E^h_t)}\mathrm{div}'(\beta\Psi')\,d\cH^d\right|\notag\\
&=\left|\int_{\mathrm{Tr}(E^h_t)}\nabla'\beta\cdot\Psi'+\beta\mathrm{div}'(\Psi')\,d\cH^d\right|\notag\\
&\sleq \left(10r_t\|\nabla'\beta\|_{L^{\infty}(\partial\Omega)}+d\right)|\mathrm{Tr}(E^h_t)|\notag\\
&\sleq 10\kappa^{-1}P(E_0)\max\{\|\nabla'\beta\|_{L^{\infty}(\partial\Omega)},d\}(r_t+1).\label{eq:term3-in-EL}
\end{align}
The last term satisfies, by the divergence theorem and the admissibility of $\Psi$, that
\begin{align}\label{eq:term4-in-EL}
\left|\lambda^h \int_{\partial^{\ast}E^h_t\cap\Omega}  \Psi\cdot\nu_{E^h_t}\,d\cH^d\right|=\left|\lambda^h\int_{E^h_t}\mathrm{div}(\Psi)\,dx\right|=|\lambda^h||E^h_t|(d+1)\sgeq\frac{m_0(d+1)}{2}|\lambda^h|.
\end{align}
Combining \eqref{eq:term1-in-EL}-\eqref{eq:term4-in-EL} with Proposition \ref{prop:EL-equation}, there is a constant $C=C(d,\kappa,P(E_0),\|\nabla'\beta\|_{L^{\infty}(\partial\Omega)})>0$ such that
\begin{align}\label{eq:r-bounds-lambda-prestep}
\frac{m_0(d+1)}{2}|\lambda^h|\sleq C(1+r_t+r_t\|v^h\|_{L^2(\partial^{\ast}E^h_t\cap\Omega,d\cH^d)}).
\end{align}
Performing integration over $[0,T]$, applying the Cauchy-Schwarz inequality and Proposition \ref{prop:L2-discrete-velocities}, we obtain that there exists a constant $C=C(d,\kappa,m_0,P(E_0),\|\nabla'\beta\|_{L^{\infty}(\partial\Omega)},T)>0$ such that
\begin{align}\label{eq:r-bounds-lambda}
\int_0^t|\lambda^h(\tau)|\,d\tau\sleq C\left(t+\left(\int_0^tr_{\tau}^2\,d\tau\right)^{1/2}\right)\qquad\text{for }t\in[0,T].
\end{align}
From \eqref{eq:lambda-bounds-r} and \eqref{eq:r-bounds-lambda},
\[
r_t\sleq r_0+C\left(t+\left(\int_0^tr_{\tau}^2\,d\tau\right)^{1/2}\right)\qquad\text{for }t\in[0,T].
\]
A standard ODE argument yields (i). Claim (ii) follows from (i) by taking square and integrating \eqref{eq:r-bounds-lambda-prestep} over $[0,T]$.
\end{proof}

\subsection{Proof of Proposition \ref{prop:spherical-caps}}\label{subsec:proof-prop10}


Throughout this section, $e_1$ denotes the vector $(1,0,\cdots,0)\in\R^{d+1}$, and $BV_{loc}(\Omega,\{0,1\})$ denotes the set of sets of locally finite perimeter, that is,
\[
BV_{loc}(\Omega,\{0,1\}):=\left\{E\subset\Omega\,:\,E\cap K\in BV(\Omega,\{0,1\})\quad \text{for every compact set }K\right\}.
\]

For $E\in BV_{loc}(\Omega,\{0,1\})$, an open subset $W$ of $\Omega$, and an elliptic integrand $\Phi\,:\,\overline{\Omega}\times\R^{d+1}\to\R_{\sgeq0}$ (see \cite[Definition 1.1]{DePM15}), we define
\[
\Phi(E,W):=\int_{\partial^{\ast}E\cap W}\Phi(x,\nu_E(x))\,d\cH^d(x)\in[0,\infty].
\]
When $W=\Omega$, we simply write $\Phi(E,\Omega)=\Phi(E)$.

\begin{defn}[\!\!\protect{\cite[Definition 2.3]{DePM15}, sub/super/minimizer}]
Let $E\in BV_{loc}(\Omega,\{0,1\})$, and let $\Phi\,:\,\overline{\Omega}\times\R^{d+1}\to\R_{\sgeq0}$ be an elliptic integrand.
\begin{itemize}
    \item[(i)] The set $E$ is called a subminimizer of $\Phi$ in $\Omega$ if
    \[
    \Phi(E,W)\sleq\Phi(F,W)
    \]
    whenever $F\subset E$ and $E\setminus F\subset\subset W$ for some open and bounded set $W$.
    \item[(ii)] The set $E$ is called a superminimizer of $\Phi$ in $\Omega$ if
    \[
    \Phi(E,W)\sleq\Phi(F,W)
    \]
    whenever $E\subset F\subset \Omega$ and $F\setminus E\subset\subset W$ for some open and bounded set $W$.
    \item[(iii)] The set $E$ is called a minimizer of $\Phi$ in $\Omega$ if
    \[
    \Phi(E,W)\sleq\Phi(F,W)
    \]
    whenever $F\subset \Omega$ and $E\Delta F\subset\subset W$ for some open and bounded set $W$.
\end{itemize}
\end{defn}

\begin{remark}\label{rmk:minimizers}
The following collect preliminary facts about minimizers from \cite{DePM15}.
\begin{itemize}
\item[(i)] It is true that $E$ is a minimizer if and only if $E$ is both a subminimizer and a superminimizer.
\item[(ii)] If $E$ is a minimizer of a given elliptic integrand $\Phi=\Phi(\nu)$ (see \cite[Definition 1.1]{DePM15}) that is homogeneous over the spatial variable, then the rescaled set $\frac{1}{r}(E-x_0):=\{\frac{1}{r}(x-x_0)\,:\,x\in E\}$ is also a minimizer of $\Phi$.
\item[(iii)] If $\{E_k\}_{k\in\N}$ is a sequence of minimizers of an elliptic integrand $\Phi$, there exist a subsequence $\{E_{k_j}\}_{j\in\N}$ and a minimizer $E$ of $\Phi$ such that
\[
E_{k_j}\to E\qquad\text{in }L^1_{loc}(\R^{d+1})\text{ as }j\to\infty.
\]
\item[(iv)] Lower density estimates on contact set hold \cite[Lemma 2.15, Lemma 2.16]{DePM15}: Let $\Phi$ be an elliptic integrand that is homogeneous over the spatial variable and satisfies $\lambda^{-1}\sleq\Phi\sleq\lambda$ for $\lambda\sgeq1$ for unit vectors. For a minimizer $E$ of $\Phi$ and $x\in\partial E\cap\partial\Omega$, there exists $r_0=r_0(d,\lambda),\,c=c(d,\lambda)>0$ such that
\[
\cH^d(B_r(x)\cap \partial E\cap\partial\Omega)\sgeq cr^d\qquad\text{for }r\in(0,r_0).
\]
\end{itemize}
\end{remark}

Throughout this section, we set and fix the elliptic integrand
\[
\Phi(x,\nu):=|\nu|+\beta(x')\nu\cdot e_{d+1}\qquad\text{for }(x,\nu)\in\overline{\Omega}\times\R^{d+1}.
\]

\begin{lemma}
For $E\in BV_{loc}(\Omega,\{0,1\})$, it holds that
\[
C_{\beta}(E)=\Phi(E).
\]
\end{lemma}
\begin{proof}
It suffices to show the lemma for a bounded set $E\in BV(\Omega,\{0,1\})$. We, first of all, have
\begin{align*}
C_{\beta}(E)&=\int_{\partial^{\ast}E\cap\Omega}\,d\cH^d+\int_{\mathrm{Tr(E)}}\beta\,d\cH^d\\
&=\int_{\partial^{\ast}E\cap\Omega}|\nu_E|\,d\cH^d+\int_{\mathrm{Tr(E)}}\beta(-e_{d+1})\cdot(-e_{d+1})\,d\cH^d
\end{align*}
By the divergence theorem,
\begin{align*}
\int_{\mathrm{Tr(E)}}\beta(-e_{d+1})\cdot(-e_{d+1})\,d\cH^d=-\int_{\partial^{\ast}E\cap\Omega}\beta(-e_{d+1})\cdot\nu_E\,d\cH^d=\int_{\partial^{\ast}E\cap\Omega}\beta\nu_E\cdot e_{d+1}\,d\cH^d,
\end{align*}
since the vector field $\beta(x')e_{d+1}$ is divergence-free. Therefore,
\begin{align*}
C_{\beta}(E)&=\int_{\partial^{\ast}E\cap\Omega}|\nu_E|\,d\cH^d+\int_{\partial^{\ast}E\cap\Omega}\beta\nu_E\cdot e_{d+1}\,d\cH^d=\Phi(E).
\end{align*}
\end{proof}

The following anisotropic Young's law is established in \cite{DePM15}.

\begin{proposition}[\!\!\protect{\cite[Proposition 2.6]{DePM15}}]\label{prop:anisotropic-young's-law}
Let $F=\{x\in\Omega\,:\,x\cdot\nu<c\}$, $\nu\in\bS^d\setminus\{\pm e_{d+1}\}$, $c\in\R$. Let $x_0\in\partial\Omega$ and $\Phi_{x_0}=\Phi(x_0,\cdot)$. Then,
\begin{itemize}
    \item[(i)] $F$ is a subminimizer of $\Phi_{x_0}$ if and only if $\nabla\Phi_{x_0}(\nu)\cdot e_{d+1}\sleq0$
    \item[(ii)] $F$ is a superminimizer of $\Phi_{x_0}$ if and only if $\nabla\Phi_{x_0}(\nu)\cdot e_{d+1}\sgeq0$.
\end{itemize}
\end{proposition}

\begin{figure}[htbp]
	\begin{center}
            \includegraphics[height=5.25cm]{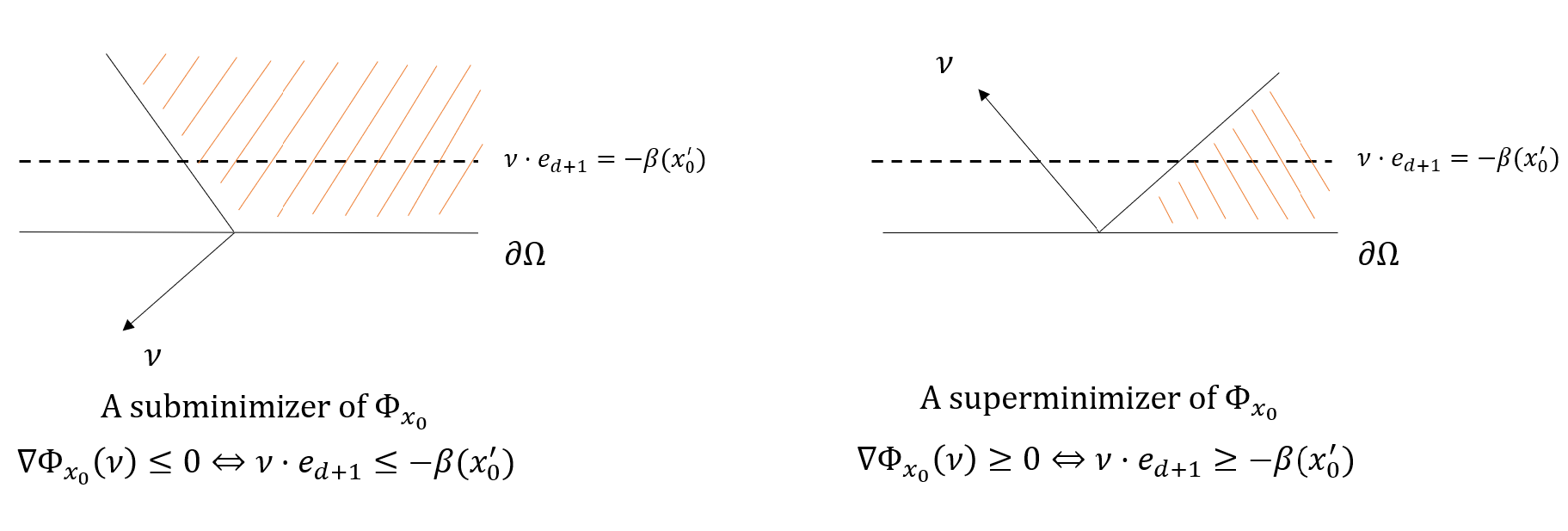}
		\vskip 0pt
		\caption{A subminimizer and a superminimizer.}
        \label{fig:sub/superminimizer}
	\end{center}
\end{figure}

For $E\in BV_{loc}(\Omega,\{0,1\})$ and $x_0\in\overline{\Omega}$, we define the set $\cB_{x_0}(E)$ of blowups of $E$ at $x_0$ as
\[
\mathcal{B}_{x_0}(E) = \left\{ F \subset \R^{d+1}\, :\, \text{there exists }r\to0\text{ such that }\frac{1}{r}(E-x_0)\to F\text{ in }L^1_{loc}(\R^{d+1})\text{ as }r\to0 \right\}.
\]
Here, $\frac{1}{r}(E-x_0)$ denotes the set $\{\frac{1}{r}(x-x_0)\,:\,x\in E\}$. We note, by a diagonal argument, that
\begin{equation*}
\begin{cases}
\text{$\mathcal{B}_{x_0}(E)$ is a compact subset of $L^1_{loc}(\R^{d+1})$},\qquad\text{and}\\
\cB_0(F)\subset\cB_{x_0}(E)\qquad\text{for }F\in\cB_{x_0}(E).
\end{cases}
\end{equation*}

\begin{lemma}[\!\!\protect{\cite[Lemma 5.2]{DePM15}}]
Let $E\in BV_{loc}(\Omega,\{0,1\})$ be a minimizer of $\Phi$, and let $x_0\in\partial_{\partial\Omega}\left(\partial E\cap\partial\Omega\right)$. Then,
\begin{itemize}
    \item[(i)] $\cB_{x_0}(E)\neq\emptyset$, and $\emptyset,\Omega\notin\cB_{x_0}(E)$.
    \item[(ii)] Every element of $\cB_{x_0}(E)$ is a minimizer of $\Phi_{x_0}$.
\end{itemize}
\end{lemma}

Let $E_0\in BV(\Omega,\{0,1\})$ be a bounded set of volume  $m_0$, and let $\{E^h_t\}_{t\sgeq0}$ be an approximate flat flow for \eqref{eq:VPMCF-contact-angle} with initial datum $E_0$. For $t\sgeq 0$, let $r_t:=\inf\{r>0\,:\,E^h_t\subset C_r\}$ as in Proposition \ref{prop:spherical-caps}. To prove Proposition \ref{prop:spherical-caps}, we assume for the contrary that there exist $t\sgeq h$ and $x_0\in\partial E^h_t\cap\partial C\cap\partial\Omega$.

For simplicity, we let $C:=C_{r_t}$, and let $\nu_C$ be the outer unit normal vector $\nu_C$ of $\partial C$ at $x_0\in\partial C\cap\partial\Omega$. By the construction of $C$, we have
\[
\nu_C\cdot e_{d+1}=1-\kappa.
\]
By rotation, we can assume without loss of generality that there exists $\alpha_C\in\left(-\frac{\pi}{2},\frac{\pi}{2}\right)$ such that
\begin{align*}
\nu_C=\cos(\alpha_C)(-e_1)-\sin(\alpha_C)e_{d+1}.
\end{align*}
Let
\[
H_C:=\left\{x\in\Omega\,:\,x\cdot\nu_C\sleq x_0\cdot\nu_C\right\}.
\]
By the choice of $r_t$, we have inclusion relations and
\begin{equation}
\begin{cases}
E^h_t\subset C\subset H_C,\\
\partial E^h_t\cap\partial\Omega\subset\left\{x\in\partial\Omega\,:\,x\cdot(-e_1)\sleq x_0\cdot(-e_1)\right\}.
\end{cases}\label{eq:inclusions}
\end{equation}
\begin{figure}[htbp]
	\begin{center}
            \includegraphics[height=5.5cm]{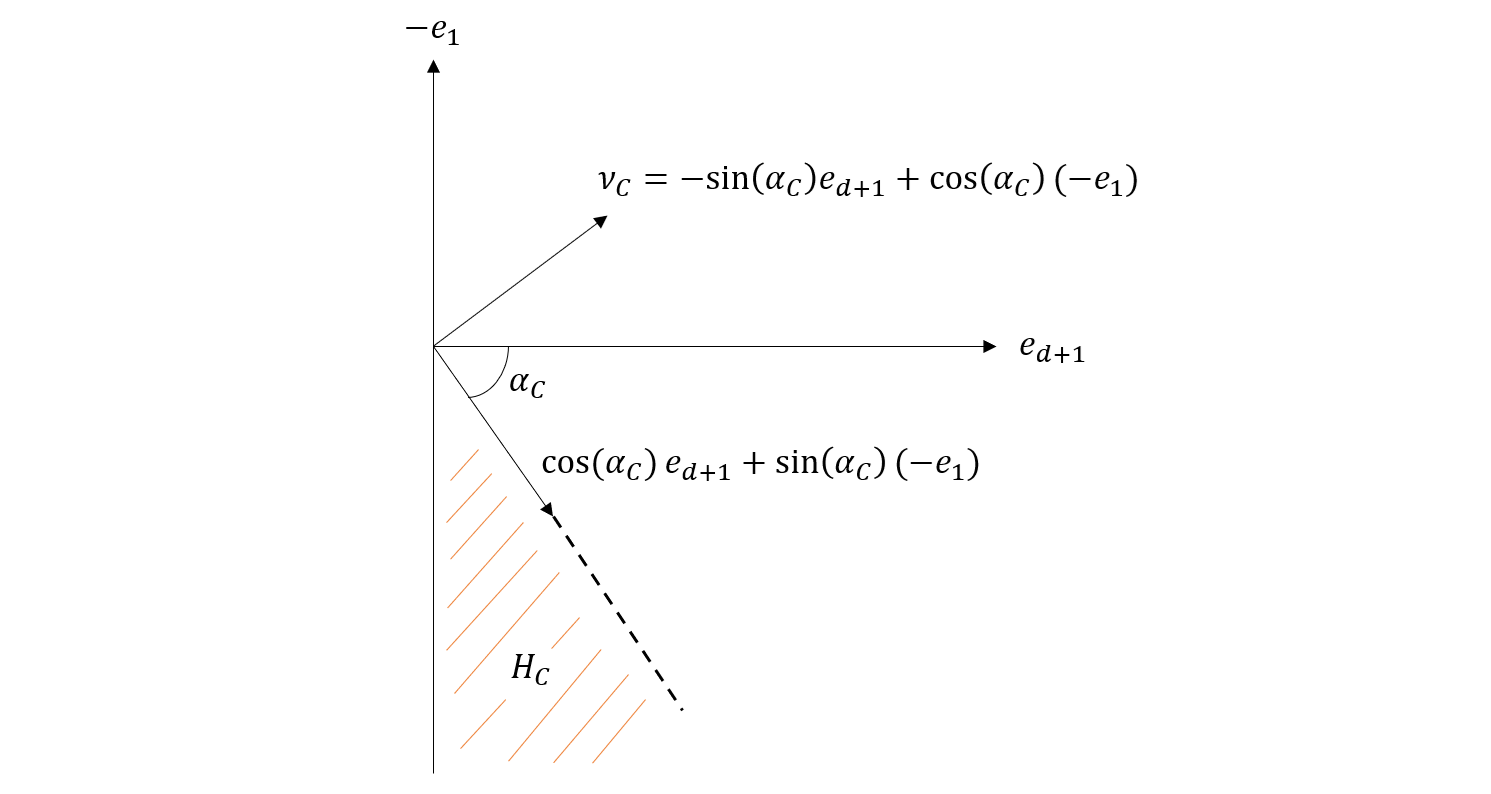}
		\vskip 0pt
		\caption{The normal vector $\nu_C$ and the space $H_C$ (with $x_0=0$).}
        \label{fig:spherical-cap}
	\end{center}
\end{figure}

\begin{lemma}[\!\!\protect{\cite[Lemma 5.3]{DePM15}}]\label{lem:wedging}
Suppose that there exists $x_0\in\partial E^h_t\cap\partial C\cap\partial\Omega$. Then,
\begin{itemize}
    \item[(i)] Every $F\in\cB_{x_0}(E^h_t)$ satisfies $F\subset\left\{x\in\overline{\Omega}\,:\,x\cdot\nu_C\sleq0\right\}$.
    \item[(ii)] Every $F\in\cB_{x_0}(E^h_t)$ satisfies
    \[
    \cH^d\left(\partial F\cap\partial\Omega\setminus\left\{x\in\partial\Omega\,:\,x\cdot(-e_1)\sleq0\right\}\right)=0.
    \]
    \item[(iii)] There exists $L>0$ such that for any $F\in\cB_{x_0}(E^h_t)$,
    \begin{equation*}
    \begin{cases}
    \sup\left\{\frac{x\cdot(-e_1)}{x\cdot e_{d+1}}\,:\,x\in\partial F\cap\Omega\right\}\sleq -\frac{1-\kappa}{\sqrt{2\kappa-\kappa^2}},\qquad\text{and}\\
    \inf\left\{\frac{x\cdot(-e_1)}{x\cdot e_{d+1}}\,:\,x\in\partial F\cap\Omega\right\}\sgeq-L.
    \end{cases}
    \end{equation*}
\end{itemize}
\end{lemma}
\begin{proof}
The first two claims follow from \eqref{eq:inclusions}, and we focus on (iii). We note that (i) implies the first part of (iii) from the fact that $\tan(\alpha_C)=-\frac{1-\kappa}{\sqrt{2\kappa-\kappa^2}}$.

We claim that (ii) implies the second part of (iii), that is, there exists $L>0$ such that
\[
\inf\left\{\frac{x\cdot(-e_1)}{x\cdot e_{d+1}}\,:\,x\in\partial F\cap\Omega\right\}\sgeq -L\qquad\text{for any }F\in\cB_{x_0}(E^h_t).
\]
Assume for the contrary that there exist $F_k\in\cB_{x_0}(E^h_t),\,x_k\in\partial F_k\cap\Omega$ for each $k\in\N$ such that
\begin{equation}
\begin{cases}
\cH^d\left(\partial F_k\cap\partial\Omega\setminus\left\{x\in\partial\Omega\,:\,x\cdot(-e_1)\sleq0\right\}\right)=0,\qquad\text{and}\\
\frac{x_k\cdot(-e_1)}{x_k\cdot e_{d+1}}\to-\infty\qquad\text{as }k\to\infty.
\end{cases}\label{eq:negative-infinity}
\end{equation}
By translations in directions in $e_1^{\perp}\cap e_{d+1}^{\perp}$, we can assume without loss of generality that
\[
x_k=(x_k\cdot e_1,0,\cdots,0,x_k\cdot e_{d+1}).
\]
By Remark \ref{rmk:minimizers}(ii) and using the fact that $x_k\cdot e_1>0$, the set $G_k:=\frac{1}{x_k\cdot e_1}F_k$ is a minimizer of $\Phi_{x_0}$,
\begin{equation*}
\begin{cases}
\cH^d\left(\partial G_k\cap\partial\Omega\setminus\left\{x\in\partial\Omega\,:\,x\cdot(-e_1)\sleq0\right\}\right)=0,\qquad\text{and}\\
\left(1,0,\cdots,0,\frac{x_k\cdot e_{d+1}}{x_k\cdot e_1}\right)\in\partial G_k\cap \Omega.
\end{cases}
\end{equation*}
By Remark \ref{rmk:minimizers}(iii) and \eqref{eq:negative-infinity}, there exists a $L^1_{loc}(\R^{d+1})$-subsequential limit $G_{\infty}$ of $G_k$ such that $G_{\infty}$ is a minimizer of $\Phi_{x_0}$,
\begin{equation*}
\begin{cases}
\cH^d\left(\partial G_{\infty}\cap\partial\Omega\setminus\left\{x\in\partial\Omega\,:\,x\cdot(-e_1)\sleq0\right\}\right)=0,\qquad\text{and}\\
p:=\left(1,0,\cdots,0\right)\in\partial G_{\infty}\cap \partial\Omega.
\end{cases}
\end{equation*}
This leads to a contradiction, as the fact that $p\in\left\{x\in\partial\Omega\,:\,x\cdot(-e_1)\sleq0\right\}$ and Remark \ref{rmk:minimizers}(iv) imply
\[
\cH^d(B_r(p)\cap\partial G_{\infty}\cap\partial\Omega)>0\qquad\text{for small }r\in(0,1).
\]
Therefore, there exists a constant $L>0$ such that the claim (iii) holds true, and as the above argument works for any minimizer of $\Phi_{x_0}$, we conclude that $L$ depends only on $d,\kappa,x_0$.
\end{proof}
\begin{remark}
By running the above argument with varying elliptic integrands $\Phi_k$ for $k\in\N$, it can be concluded that the constant $L$ in Lemma \ref{lem:wedging}(iii) actually depends only on $d,\kappa$, see \cite[Definition 1.1, Lemma 5.3]{DePM15}. However, having a finite value $L>0$ is enough for our purpose, and hence, we have fixed $\Phi_{x_0}(\nu):=|\nu|+\beta(x_0')\nu\cdot e_{d+1}$ in this section.
\end{remark}

We now prove Proposition \ref{prop:spherical-caps}.

\begin{proof}[Proof of Proposition \ref{prop:spherical-caps}]
Let $E_0\in BV(\Omega,\{0,1\})$ be a bounded set of volume  $m_0$, and let $\{E^h_t\}_{t\sgeq0}$ be an approximate flat flow for \eqref{eq:VPMCF-contact-angle} with initial datum $E_0$. For $t\sgeq 0$, let $r_t:=\inf\{r>0\,:\,E^h_t\subset C_r\}$ as in Proposition \ref{prop:spherical-caps}. To prove Proposition \ref{prop:spherical-caps}, we assume for the contrary that there exist $t\sgeq h$ and $x_0\in\partial E^h_t\cap\partial C\cap\partial\Omega$. By translations, we can assume $x_0=0$ without loss of generality.

\vspace{\medskipamount}

Consider the function $\beta_1\,:\,\cB_{0}(E^h_t)\to\left[-L,-\frac{1-\kappa}{\sqrt{2\kappa-\kappa^2}}\right]$ defined by
\[
\beta_1(F):=\sup_{\partial F\cap\Omega}\frac{x\cdot(-e_1)}{x\cdot e_{d+1}}\qquad\text{for }F\in\cB_0(E^h_t).
\]
Then, $\beta_1$ is lower semi-continuous in $L^1_{loc}(\R^{d+1})$. Since $\cB_0(E^h_t)$ is a compact subset of $L^1_{loc}(\R^{d+1})$, we can find $F_1\in\cB_0(E^h_t)$ that minimizes $\beta_1$ over $\cB_0(E^h_t)$. By Lemma \ref{lem:wedging}(i), we have
\begin{align}\label{eq:F1-in-H_C}
F_1\subset\left\{x\in\overline{\Omega}\,:\,x\cdot\nu_C\sleq0\right\}.
\end{align}
Choose $\alpha_1\in\left(-\frac{\pi}{2},\frac{\pi}{2}\right)$ such that $\tan(\alpha_1)=\beta_1(F_1)$. Let
\begin{equation*}
\begin{cases}
\nu_1:=\cos(\alpha_1)(-e_1)-\sin(\alpha_1)e_{d+1},\\
H_1:=\left\{x\in\Omega\,:\,x\cdot\nu_1\sleq0\right\}.
\end{cases}
\end{equation*}
From \eqref{eq:F1-in-H_C}, we also have
\begin{align}\label{eq:H1-in-H_C}
H_1\subset H_C.
\end{align}

\vspace{\medskipamount}

We similarly consider the function $\beta_2\,:\,\cB_{0}(F_1)\to\left[-L,-\frac{1-\kappa}{\sqrt{2\kappa-\kappa^2}}\right]$ defined by
\[
\beta_2(F):=\inf_{\partial F\cap\Omega}\frac{x\cdot(-e_1)}{x\cdot e_{d+1}}\qquad\text{for }F\in\cB_0(F_1).
\]
Then, $\beta_2$ is upper semi-continuous in $L^1_{loc}(\R^{d+1})$. Since $\cB_0(F_1)$ is a compact subset of $L^1_{loc}(\R^{d+1})$, we can find $F_2\in\cB_0(F_1)$ that maximizes $\beta_2$ over $\cB_0(F_1)$. Choose $\alpha_2\in\left(-\frac{\pi}{2},\frac{\pi}{2}\right)$ such that $\tan(\alpha_2)=\beta_2(F_2)$. Let
\begin{equation*}
\begin{cases}
\nu_2:=\cos(\alpha_2)(-e_1)-\sin(\alpha_2)e_{d+1},\\
H_2:=\left\{x\in\Omega\,:\,x\cdot\nu_2\sleq0\right\}.
\end{cases}
\end{equation*}
From the inclusion \eqref{eq:F1-in-H_C}, we have $F_2\subset \overline{H_1}$, and therefore, from \eqref{eq:H1-in-H_C},
\begin{align}\label{eq:H2-in-H_C}
H_2\subset H_1\subset H_C
\end{align}
Moreover, by following the proof of \cite[Lemma 5.4]{DePM15} together with \cite[Proposition 2.5]{DePM15}, we see that $H_2$ is a subminimizer of $\Phi_0$.

\vspace{\medskipamount}

We now compare the angles $\alpha_C,\,\alpha_1,\,\alpha_2\in\left(-\frac{\pi}{2},\frac{\pi}{2}\right)$. By \eqref{eq:H2-in-H_C}, we have $\alpha_2\sleq\alpha_1\sleq\alpha_C$, and therefore,
\begin{align}\label{eq:sin-comparison}
\sin(\alpha_2)\sleq\sin(\alpha_C).
\end{align}
By Proposition \ref{prop:anisotropic-young's-law} and the fact that $H_2$ is a subminimizer of $\Phi_0$, it holds $-\sin(\alpha_2)\sleq-\beta(0)$. By the construction of the spherical cap $C$, we have $-\sin(\alpha_C)=1-\kappa$. Hence, by \eqref{eq:sin-comparison},
\begin{align*}
1-\kappa=-\sin(\alpha_C)\sleq-\sin(\alpha_2)\sleq-\beta(0)\sleq1-2\kappa,
\end{align*}
which is a contradiction.
\end{proof}

\section{Proof of Theorem \ref{thm:BV-distributional-solutions}}\label{sec:BV-solutions}
Before we prove Theorem \ref{thm:BV-distributional-solutions}, we state the following error estimate that is similar to the ones in \cite{LS95,MSS16,BK18}. As the proof is obtained in the same way by following \cite[Section 4]{MSS16} and the proof of \cite[Proposition 8.11]{BK18} (together with Proposition \ref{prop:barrier-functions}), we omit the proof.

\begin{proposition}\label{prop:error-estimate}
If $1\sleq d\sleq6$, it holds that
\[
\lim_{h \to 0} \left| \int_h^{\infty} \frac{1}{h} \left[ \int_{E_t^{h}\cap\Omega} \phi \, dx - \int_{E_{t-h}^{h}\cap\Omega} \phi \, dx \right] dt - \int_h^{\infty} \int_{\partial E_t^{h}\cap\Omega} \phi v^h \, d\mathcal{H}^{d} dt \right| = 0
\]
for every $\phi\in C^1_c(\Omega\times[0,\infty))$.
\end{proposition}

We prove Theorem \ref{thm:BV-distributional-solutions}.

\begin{proof}[Proof of Theorem \ref{thm:BV-distributional-solutions}]
(i) We first obtain from the general convergence theorem \cite[Theorem 4.4.2]{H86} (together with \eqref{eq:energy-assumption}, Proposition \ref{prop:L2-discrete-velocities}, Proposition \ref{prop:barrier-functions}, and Corollary \ref{cor:L2-mean-curvature}) that there exist functions $v:\Omega\times(0,\infty)\to\R$, $\lambda:(0,\infty)\to\R$, $H_E:\Omega\times(0,\infty)\to\R$ and constants $C(d,\kappa),\ R_T=R_T(d,\kappa,T)>0$ (for a given $T>0$) such that \eqref{eq:L2-target-functions} holds with convergence properties
\begin{equation}\label{eq:convergence-functions}
\left\{
\begin{aligned}
\int_0^\infty \int_{\partial^{\ast}E^h_t\cap\Omega} \phi v^h \, d\mathcal{H}^d \, dt\quad&\longrightarrow\quad\int_0^\infty \int_{\partial^{\ast}E_t\cap\Omega} \phi v \, d\mathcal{H}^d \, dt,\\
\int_0^\infty \int_{\partial^{\ast}E^h_t\cap\Omega} \phi \lambda^h \, d\mathcal{H}^d \, dt\quad&\longrightarrow\quad\int_0^\infty \int_{\partial^{\ast}E_t\cap\Omega} \phi \lambda \, d\mathcal{H}^d \, dt\\
\int_0^\infty \int_{\partial^{\ast}E^h_t\cap\Omega} \phi H_{E^h_t} \, d\mathcal{H}^d \, dt\quad&\longrightarrow\quad\int_0^\infty \int_{\partial^{\ast}E_t\cap\Omega} \phi  H_{E}(\cdot,t) \, d\mathcal{H}^d \, dt
\end{aligned}
\right.
\end{equation}
as $h\to0$ for $\phi\in C_c(\Omega\times[0,\infty))$, and
\begin{equation}
\left\{
\begin{aligned}\label{eq:convergence-vector-fields}
\int_0^\infty \int_{\partial^{\ast}E^h_t\cap\Omega} v^h\nu_{E^h_t}\cdot\Psi \, d\mathcal{H}^d \, dt\quad&\longrightarrow\quad\int_0^\infty \int_{\partial^{\ast}E_t\cap\Omega} v\nu_{E_t}\cdot\Psi \, d\mathcal{H}^d \, dt,\\
\int_0^\infty \int_{\partial^{\ast}E^h_t\cap\Omega} \lambda^h\nu_{E^h_t}\cdot\Psi \, d\mathcal{H}^d \, dt\quad&\longrightarrow\quad\int_0^\infty \int_{\partial^{\ast}E_t\cap\Omega} \lambda\nu_{E_t}\cdot\Psi \, d\mathcal{H}^d \, dt,\\
\int_0^\infty \int_{\partial^{\ast}E^h_t\cap\Omega} H_{E^h_t}\nu_{E^h_t}\cdot\Psi \, d\mathcal{H}^d \, dt\quad&\longrightarrow\quad\int_0^\infty \int_{\partial^{\ast}E_t\cap\Omega} H_{E}(\cdot,t)\nu_{E_t}\cdot\Psi \, d\mathcal{H}^d \, dt
\end{aligned}
\right.
\end{equation}
\noindent as $h\to0$ for admissible $\Psi\in C^1_c\left(\overline{\Omega}\times[0,\infty),\R^{d+1}\right)$. Moreover, for a.e. $t\sgeq0$, $H_E(\cdot,t)$ is a generalized mean curvature of $\partial^{\ast}E_t\cap\Omega$. Indeed, fix a $\eta\in C_c([0,\infty))$ and an admissible $\Psi\in C^1_c(\overline{\Omega},\R^{d+1})$, and let $F\in C_c(\R^{d+1}\times\R^{d+1})$ such that $F(x,\nu):=(\mathrm{id}-\nu \otimes \nu) : \nabla\Psi(x)$ on $\overline{\Omega}\times\left\{\nu\in\R^{d+1}\,:\,|\nu|\sleq2\right\}$. We then have, by \eqref{eq:energy-assumption} (see \cite[(4.2)]{MSS16}) and \eqref{eq:convergence-vector-fields}, that
\begin{align*}
\int_0^\infty \eta(t) \int_{\partial^{\ast}E_t\cap\Omega} (\mathrm{id}-\nu_{E_{t}} \otimes \nu_{E_{t}}) : \nabla \Psi \, d\mathcal{H}^{d} dt &= \int_0^\infty  \int_{\partial^{\ast}E_t\cap\Omega} \eta(t)F(x,\nu_{E_t}(x)) \, d\mathcal{H}^{d}(x) dt \\
&= \lim_{h \to 0} \int_0^\infty  \int_{\partial^{\ast}E^h_t\cap\Omega} \eta(t)F(x,\nu_{E^h_t}(x)) \, d\mathcal{H}^{d}(x) dt \\
&= \lim_{h \to 0} \int_0^\infty \int_{\partial^{\ast}E^h_t\cap\Omega} \eta(t) (\mathrm{id}-\nu_{E^h_{t}} \otimes \nu_{E^h_{t}}) : \nabla \Psi \, d\mathcal{H}^{d} dt \\
&= \lim_{h \to 0} \int_0^\infty \int_{\partial^{\ast}E^h_t\cap\Omega} \eta(t) H_{E^h_t} \nu_{E_t^h} \cdot \Psi \, d\mathcal{H}^{d} dt \\
&= \int_0^\infty \int_{\partial^{\ast}E_t\cap\Omega} \eta(t) H_{E}(\cdot,t) \nu_{E_t} \cdot \Psi \, d\mathcal{H}^{d} dt.
\end{align*}
As $\eta\in C_c([0,\infty))$ and $\Psi\in C^1_c(\overline{\Omega},\R^{d+1})$ were arbitrary, the set $\partial^{\ast}E_t\cap\Omega$ has a generalized mean curvature $H_E(\cdot,t)$ for a.e. $t\sgeq0$.

\vspace{\medskipamount}

(ii) By the divergence theorem, we have that for $\varphi\in C^1_c(\R^{d+1},\R^{d+1})$,
\[
\int_{E^h_t} \mathrm{div}(\varphi) \, dx - \int_{\partial^{\ast} E^h_t\cap\Omega} \varphi \cdot \nu_{E^h_t} \ d\cH^d = -\int_{\mathrm{Tr}(E^h_t)} \varphi \cdot e_{d+1} \ d\cH^{d}.
\]
Letting $h\to0$, and using \eqref{eq:energy-assumption} (in the form of \cite[(4.2)]{MSS16} with a function $F\in C_c(\R^{d+1}\times\R^{d+1})$ such that $F(x,\nu)=\varphi(x)\cdot\nu$ on a compact set containing $B_{R_T}$ from Proposition \ref{prop:barrier-functions}), we have
\[
\int_{\partial^{\ast} E^h_t\cap\Omega} \varphi \cdot \nu_{E^h_t} \ d\cH^d\quad\longrightarrow\quad\int_{\partial^{\ast} E_t\cap\Omega} \varphi \cdot \nu_{E_t} \ d\cH^d\qquad\text{as }h\to0.
\]
As we also have $E^h_t\to E_t$ in $L^1(\Omega)$ as $h\to0$ from Theorem \ref{thm:flat-flows}, we obtain
\[
\int_{E_t} \mathrm{div}(\varphi) \, dx - \int_{\partial^{\ast} E_t\cap\Omega} \varphi \cdot \nu_{E^h_t} \ d\cH^d = -\lim_{h\to0}\int_{\mathrm{Tr}(E^h_t)} \varphi \cdot e_{d+1} \ d\cH^{d}.
\]
By the divergence theorem now applied to $E_t$, we have
\begin{align}\label{eq:trace-convergence}
\int_{\mathrm{Tr}(E^h_t)} \phi  \ d\cH^{d}\quad\longrightarrow\quad\int_{\mathrm{Tr}(E_t)} \phi  \ d\cH^{d}\qquad\text{as }h\to0.
\end{align}
Here, we altered $\varphi\cdot e_{d+1}$ for an arbitrary vector field $\varphi\in C^1_c(\R^{d+1},\R^{d+1})$ to an arbitrary function $\phi\in C^1_c(\R^d)$.

By a standard approximation argument together with Proposition \ref{prop:dissipations} and Proposition \ref{prop:barrier-functions}, we see that \eqref{eq:trace-convergence} extends to $\phi\in L^{\infty}(\partial\Omega)$. Indeed, for a given $T>0$ and $t\in[0,T]$, we choose $R_T>0$ as in Proposition \ref{prop:barrier-functions}. Let $\phi\in L^{\infty}(\partial\Omega)$ and let $\{\phi_k\}_k\in\N$ be a sequence of $C^{\infty}(\partial\Omega)$-functions such that $\|\phi_k-\phi\|_{L^{\infty}(\partial\Omega)}\sleq\frac
1k$. In the next argument, we can assume that $\phi,\,\phi_k$ are supported in $B_{R_T}$ in the following without loss of generality. Write
\begin{align*}
\left|\int_{\mathrm{Tr}(E^h_t)}\phi\,d\cH^ddt-\int_{\mathrm{Tr}(E_t)}\phi\,d\cH^ddt\right|&\sleq\left|\int_{\mathrm{Tr}(E^h_t)}\phi\,d\cH^ddt-\int_{\mathrm{Tr}(E^h_t)}\phi_k\,d\cH^ddt\right|\\
&\qquad\qquad+\left|\int_{\mathrm{Tr}(E^h_t)}\phi_k\,d\cH^ddt-\int_{\mathrm{Tr}(E_t)}\phi_k\,d\cH^ddt\right|\\
&\qquad\qquad\qquad\qquad+\left|\int_{\mathrm{Tr}(E_t)}\phi_k\,d\cH^ddt-\int_{\mathrm{Tr}(E_t)}\phi\,d\cH^ddt\right|\\
&\sleq\frac{|\mathrm{Tr}(E^h_t)|+|\mathrm{Tr}(E_t)|}{k}+\left|\int_{\mathrm{Tr}(E^h_t)}\phi_k\,d\cH^ddt-\int_{\mathrm{Tr}(E_t)}\phi_k\,d\cH^ddt\right|\\
&\sleq\frac{2\cH^d(B_{R_T}\cap\partial\Omega)}{k}+\left|\int_{\mathrm{Tr}(E^h_t)}\phi_k\,d\cH^ddt-\int_{\mathrm{Tr}(E_t)}\phi_k\,d\cH^ddt\right|.
\end{align*}
Taking the limit-supremum as $h\to0$ and letting $k\to\infty$ in order imply \eqref{eq:trace-convergence} for arbitrary $\phi\in L^{\infty}(\partial\Omega)$.

\vspace{\medskipamount}

Let $\eta\in C^1_c([0,\infty))$ and admissible $\Psi\in C^1_c(\overline{\Omega},\R^{d+1})$. Write, from \eqref{eq:discrete-EL}, that
\begin{align*}
\int_0^{\infty}\eta(t)&\int_{\partial^{\ast}E^h_t\cap\Omega} \mathrm{div}_{\partial E^h_t}\Psi\,d\cH^ddt + \int_0^{\infty}\eta(t)\int_{\partial^{\ast}E^h_t\cap\Omega} v^h\Psi\cdot\nu_{E^h_t}\,d\cH^ddt\notag\\
&+\int_0^{\infty}\eta(t)\int_{\mathrm{Tr}(E^h_t)}\mathrm{div'}(\beta\Psi')\,d\cH^{d}dt = \int_0^{\infty}\eta(t)\lambda^h(t) \int_{\partial^{\ast}E^h_t\cap\Omega}  \Psi\cdot\nu_{E^h_t}\,d\cH^ddt.
\end{align*}
We apply \eqref{eq:convergence-functions}, \eqref{eq:convergence-vector-fields} (by viewing $\eta(t)\Psi(x)$ as a admissible test vector field in $C_c^1(\overline{\Omega}\times[0,\infty))$), \eqref{eq:trace-convergence} with the dominated convergence theorem, we have
\begin{align*}
\int_0^{\infty}\eta(t)&\int_{\partial^{\ast}E_t\cap\Omega} \mathrm{div}_{\partial E_t}\Psi\,d\cH^ddt + \int_0^{\infty}\eta(t)\int_{\partial^{\ast}E_t\cap\Omega} v\Psi\cdot\nu_{E_t}\,d\cH^ddt\notag\\
&+\int_0^{\infty}\eta(t)\int_{\mathrm{Tr}(E_t)}\mathrm{div'}(\beta\Psi')\,d\cH^{d}dt = \int_0^{\infty}\eta(t)\lambda(t) \int_{\partial^{\ast}E_t\cap\Omega}  \Psi\cdot\nu_{E_t}\,d\cH^ddt.
\end{align*}
Since $\eta\in C^1_c([0,\infty))$ was arbitrary, we obtain \eqref{eq:without-perimeter-assumption} for a.e. $t\sgeq0$.

\vspace{\medskipamount}

(iii) Let $\phi\in C^1_c(\Omega\times[0,\infty))$. We write, by a change of variables and using the initial datum $E_0$,
\[
\int_h^\infty \left[ \int_{E_t^{h}} \phi \,dx - \int_{E_{t-h}^{h}} \phi \,dx \right] dt = \int_h^\infty \int_{E_t^{h}} (\phi(x,t) - \phi(x,t+h)) \,dxdt - h \int_{E_0} \phi \,dx,
\]
and therefore,
\[
\lim_{h \to 0} \int_h^\infty \frac{1}{h} \left[ \int_{E_t^{h}} \phi \,dx - \int_{E_{t-h}^{h}} \phi \,dx \right] dt = - \int_0^\infty \int_{E_t} \phi_t(x,t) \,dx\,dt - \int_{E_0} \phi \,dx.
\]
By the above and Proposition \ref{prop:error-estimate}, we obtain \eqref{eq:motion-law-v}. Also, \eqref{eq:motion-law-lambda} follows from \eqref{eq:pointwise} and \eqref{eq:convergence-functions}. The identity \eqref{eq:lambda} is seen by taking a cylindrical test function $\phi(x,t)=\eta(t)\zeta^r(x)$ where $\eta\in C_c^1([0,\infty))$, and $\zeta^r(x)\in C^1_c(\Omega)$ is $1$ on $[-r^{-1},r^{-1}]\times\cdots\times[-r^{-1},r^{-1}]\times[r,r^{-1}]$ for $r\in(0,1).$ Then, \eqref{eq:motion-law-v} gives
\[
\int_0^{\infty}\int_{E_t}\eta_t(t)\zeta^r(x)\,dxdt+\int_{E_0}\eta(0)\zeta^r(x)\,dx=-\int_0^{\infty}\int_{\partial^{\ast}E_t\cap\Omega}\eta(t)\zeta^r(x)\,dxdt.
\]
Letting $r\to0$, we obtain
\[
\int_0^{\infty}\eta_t(t)|E_t|\,dt+\eta(0)|E_0|=-\int_0^{\infty}\eta(t)\int_{\partial^{\ast}E_t\cap\Omega}v\,d\cH^ddt.
\]
Using the fact that $|E_t|=0$ for all $t\sgeq0$ and that $\eta\in C_c^1([0,\infty))$ was arbitrary, we see that the left-hand side is zero, and thus,
\[
\int_{\partial^{\ast}E_t\cap\Omega}v\,d\cH^d=0\qquad\text{for a.e. }t\sgeq0.
\]
Repeating the above argument with \eqref{eq:motion-law-lambda}, we get \eqref{eq:lambda} for a.e. $t\sgeq0$.
\end{proof}





\bibliographystyle{plain}
\bibliography{Preprint}

\end{document}